\documentclass[a4paper,11pt]{amsart}

\usepackage[T1]{fontenc}
\usepackage[utf8x]{inputenc}
\usepackage{lmodern}
\usepackage{enumerate}
\usepackage{url}
\usepackage{comment}
\usepackage{amsfonts,amsmath,amstext,amsbsy,amssymb,amsbsy,
  amsopn,amsthm,amscd} 
\usepackage[T1]{fontenc}
\usepackage{hyperref}

\RequirePackage[dvipsnames]{xcolor} % [dvipsnames]
\definecolor{halfgray}{gray}{0.55}%chapter numbers will be semi
\definecolor{webgreen}{rgb}{0,0.5,0}
\definecolor{webbrown}{rgb}{.6,0,0} 
\definecolor{red}{rgb}{1,0,0} 
\hypersetup{%
  colorlinks=true, linktocpage=true, pdfstartpage=3,
  pdfstartview=FitV,%
  breaklinks=true, pdfpagemode=UseNone, pageanchor=true,
  pdfpagemode=UseOutlines,%
  plainpages=false, bookmarksnumbered, bookmarksopen=true,
  bookmarksopenlevel=1,%
  hypertexnames=true,
  pdfhighlight=/O,%hyperfootnotes=true,%nesting=true,%frenchlinks,%
  urlcolor=webbrown, linkcolor=RoyalBlue,
  citecolor=webgreen, %pagecolor=RoyalBlue,%
  pdftitle={Invariant Distributions for homogeneous flows},%
  pdfauthor={L. Flaminio, G. Forni, F. Rodriguez Hertz},%
  pdfsubject={2000 MAthematical Subject Classification: Primary:
    37-XX, 37C15, 37C40},%
  pdfkeywords={Cohomological Equations, Homogeneous flows},%
  pdfcreator={pdfLaTeX},%
  pdfproducer={LaTeX with hyperref}%
}
\usepackage{mathptmx}\usepackage{microtype}
\usepackage{tikz-cd}
\usepackage{wrapfig}  
\newif\iftikz

\ifcsname macrostupidino\endcsname%
\tikzfalse%
\else%
\tikztrue%
\fi%

\newtheorem{theorem}{Theorem}[section]
\newtheorem{lemma}[theorem]{Lemma}
\newtheorem{corollary}[theorem]{Corollary}
\newtheorem{proposition}[theorem]{Proposition}
\theoremstyle{definition}
\newtheorem{problem}{Problem}[section]

\newtheorem{conjecture}{Conjecture}[section]
\newtheorem{definition}[theorem]{Definition}

\newtheorem{remark}[theorem]{Remark}
\newcommand{\field}[1]{\mathbb{#1}}
\newcommand{\R}{\field{R}}
\newcommand{\N}{\field{N}}
\def\C{\field{C}}
\newcommand{\Z}{\field{Z}}
\newcommand{\Q}{\field{Q}}
\newcommand{\T}{\field{T}}
\newcommand{\ad}{\mathrm a\mathrm d}
\newcommand{\Ad}{\mathrm A\mathrm d}
\renewcommand{\>}{{\rangle}}
\newcommand{\<}{{\langle}}
\renewcommand{\|}{\,\Vert\,}
\renewcommand{\sl}{\mathfrak s \mathfrak l}
\newcommand{\SL}{\mathrm S \mathrm L}

\def\D{\hbox{d}}

\def\Aff{\operatorname{Aff}} 
\def\Aut{\operatorname{Aut}}
\def\Ggroup{\mathcal A} 
\def\Gsubgp{\mathcal N} % normal subgroup of \Ggroup stable by \Hgroup
\def\Dlattice{\mathcal D} % quasi-lattice in \Gsubgp
\def\Hgroup{\mathcal H} % morally group of (inner) automorphims of \Ggroup
\def\hatf{f} 
\def\Gphi{\widehat{\Gsubgp}} 

\def\biglattice{\widehat{\Dlattice}}

\def\Levi{\mathcal L}
\def\flow{\phi}
\def\map{\psi}

\def\diffeo{f}
\def\finitegp{\mathcal F}
\def\susp#1{\Sigma #1}
\def\nilpt{N}
\def\finitegp{\mathcal F}

\title[Invariant Distributions for homogenoeus actions]%
{Invariant Distributions for homogeneous flows \\ and affine transformations}
\author{L.~Flaminio}
\address[L. Flaminio]{Unit\'e Mixte de Recherche CNRS 8524 \\
  Unit\'e de Formation et Re\-cher\-che de Math\'ematiques\\
  Universit\'e de Lille\\
  F59655 Villeneuve d'Asq CEDEX\\
  FRANCE} 
\email{livio.flaminio@math.univ-lille1.fr} 
\author{G.~Forni}
\address[G.~Forni]{Department of Mathematics\\
University of Maryland\\ College Park, MD~20742,~U.S.A} 
\email{gforni@math.umd.edu} 
\author{F.~Rodriguez Hertz}
\address[F.~Rodriguez Hertz]{Mathematics Department\\
Penn State University\\
University Park\\ 
State College, PA~16802, U.S.A.} 
\email{hertz@math.psu.edu} 

\begin{document}

\begin{abstract}
  \begin{sloppypar}
    We prove that every homogeneous flow on a finite-volume
    homogeneous manifold has countably many independent invariant
    distributions unless it is conjugate to a linear flow on a
    torus. We also prove that the same conclusion holds for every
    affine transformation of a homogenous space which is not conjugate
    to a toral translation. As a part of the proof, we have that any
    smooth partially hyperbolic flow on any compact manifold has
    countably many distinct minimal sets, hence countably many
    distinct ergodic probability measures. As a consequence, the Katok
    and Greenfield-Wallach conjectures hold in all of the above cases.
  \end{sloppypar}
\end{abstract}
\maketitle

\section{Introduction}
\label{sec:1}
A smooth flow $\flow_t$ generated by a smooth vector field $X$ on a
compact manifold~$M$ is called \emph{stable} if the range of the Lie
derivative $\mathcal L_X: C^\infty(M) \to C^\infty(M)$ is closed and
it is called \emph{cohomology-free} or \emph{rigid} if it is stable
and the range of the Lie derivative operator has codimension one.  For
a smooth diffeomorphism $f$ on $M$, stability and rigidity are
analogously defined by considering the range of the operator $f^*-I :
C^\infty(M) \to C^\infty(M)$.  The properties of stability and
rigidity are easily seen to be invariant under smooth conjugacies, in
the case of either flows or diffeomorphisms, and under smooth
reparametrizations, in the case of flows.

The \emph{Katok (or Katok-Hurder) conjecture} \cite{K01}, \cite{K03},
\cite{H85} for flows states that every cohomology-free smooth flow is
smoothly conjugate to a linear flow on a torus with Diophantine
frequencies.  It is not hard to prove that all coho\-mology-free flows
are volume preserving and uniquely ergodic (see for instance
\cite{Forni}). An analogous conjecture can be stated for smooth
diffeomorphisms.

We also recall that the Katok conjecture for flows is equivalent to
the \emph{Greenfield-Wallach conjecture}~\cite{GW} stating that every globally
hypoelliptic vector field is smoothly conjugate to a Diophantine
linear flow (see \cite{Forni}). A smooth vector field $X$ is
called~\emph{globally hypoelliptic} if any $0$-dimensional current $U$
on $M$ is smooth under the condition that the current $\mathcal L_X U$
is smooth. Greenfield and Wallach in~\cite{GW} proved this conjecture
for smooth flows on compact surfaces, for homogeneous flows in
dimension $3$ and for homogeneous flows on \emph{compact} Lie groups
in all dimensions.  (The equivalence of the Katok and
Greenfield-Wallach conjectures was essentially proved already
in~\cite{CC00} as noted by the third author of this paper.  The
details of the proof can be found in~\cite{Forni}).

The best general result to date in the direction of a proof is the
joint paper of the third author \cite{HH} where it is proved that
every cohomology-free vector field has a factor smoothly conjugate to
Diophantine linear flow on a torus of the dimension equal to the first
Betti number of the manifold $M$. This result has been developed
independently by several authors \cite{Forni}, \cite{Kocsard},
\cite{Matsumoto} to give a complete proof of the conjecture in
dimension~$3$ and by the first author in the joint paper
\cite{FlaminioPaternain} to prove that every cohomology-free flow can
be embedded continuously as a linear flow in a possibly non-separated
Abelian group.

From the definition, it is clear that there are two main mechanisms
which may prevent a smooth flow from being cohomology-free: it can
happen that the flow is not stable or it can happen that the closure
of its range has codimension higher than one (or both).

Examples of stable flows or diffeomorphisms with range of infinite
codimension have been known for a long time: geodesic flows on
manifolds of negative curvature and in general transitive Anosov flows
and diffeomorphisms are perhaps the oldest. In this case, by the
Livsic theorem~\cite{Liv}, the joint kernel of all invariant measures
carried by periodic orbits coincides with the closure of the range of
the Lie derivative operator in the H\"older category. The Livsic
theorem has been generalized to the smooth or analytic category in
several cases~\cite{CEG},~\cite{GK}, \cite{LMM}.  For partially
hyperbolic diffeomorphisms A.~Katok and A.~Kono\-nenko \cite{MR1386840}
introduced a family of continuous functionals on $C^\alpha$ functions,
the Periodic Cycle Functionals (PCF's).  They showed that under a
local accessibility condition a $C^\alpha$ function is a coboundary,
modulo constants, if and only if it belongs to the kernel of all
PCF's.  A.~Wilkinson~\cite{Wilkinson} generalized this result to
accessible partially hyperbolic systems. Hence these systems are
$C^\alpha$ stable. To the best of our knowledge there is no proof in
the literature that PCF's span an infinite dimensional space or even
that they are always non trivial. By an elegant argument suggested by
A.~Katok in a personal communication it is possible to derive that
PCF's span an infinite dimensional space of distributions under the
condition that the system is not uniquely ergodic and that finite
linear combinations of PCF's are never measures. In Section~2 we show
that by Wilkinson's work \cite{Wilkinson} under the accessibility
condition Theorem~\ref{expandingminimal} implies as an easy corollary
that the Periodic Cycle Functionals span an infinite dimensional space
of continuous functionals on $C^\alpha(M)$ (see
Corollary~\ref{pcf_maps}) for all partially hyperbolic $C^1$
diffeomorphisms.  A similar statement holds for flows (see
Corollary~\ref{pcf_flows}). The proof depends on an extension of
Wilkinson's theorem on solutions of the cohomological equation to
flows. To the best of our knowledge such an extension is not in the
literature, however it follows quite easily from Wilkinson's results
for maps (see Theorem~\ref{cohomeq_PHflows}).

We also note that generalizations of Livsic theorem in the \emph{non
  accessible} case are due to Veech~\cite{Veech} and
Dolgopyat~\cite{Dolgopyat}.

Among non-hyperbolic systems, in fact among systems of parabolic type,
linear toral skew-shifts \cite{K01}, translation flows \cite{Forni97},
\cite{MMY}, horocycle flows \cite{FF1} and nilflows \cite{FF} are in
general \emph{uniquely ergodic} and stable, but have range of
countable codimension in the smooth category. Translation flows are a
special case as the range closure is finite dimensional in all spaces
of finitely differentiable functions \cite{Forni97}.

Flows and diffeomorphisms on compact manifolds with range closure of
codimension one in the space of smooth functions are
called~\emph{distributionally uniquely ergodic} (DUE) (see for
instance~\cite{AFK}). The motivation for this terminology comes from
the fact that DUE flows or diffeomorphisms can also be defined by the
condition that they are uniquely ergodic and that the space of all
invariant distributions is spanned by the unique invariant probability
measure. We recall that an \emph{invariant distribution} for a flow or
a diffeomorphism is a distribution (in the sense of S.~Sobolev and
L.~Schwartz) that is invariant under the natural action by
push-forward under the flow or diffeomorphism on the space of
distributions.  In the case of flows an equivalent definition requires
that the Lie derivative of the distribution along the flow vanishes in
the sense of distributions.

Linear flows on tori with Liouvillean frequencies are examples of
non-stable DUE flows; it is remarkable that distributional unique
ergodicity may coexists with a ``chaotic'' property such as
weak-mixing. In fact B.~Fayad~\cite{Fayad} has constructed examples of
mixing smooth time-changes of Liouvillean linear flows on tori.

Until recently there were no examples of DUE systems, except for toral
systems derived from \emph{linear} Liouvillean systems. In the past
few years several new examples of this kind have been found by
A.~Avila and collaborators.  Avila and A.~Kocsard \cite{AvilaKocsard}
have proved that all smooth circle diffeomorphisms with irrational
rotation number are DUE\@. Recently Avila, Fayad and Kocsard~\cite{AFK}
have constructed examples of DUE flows on certain higher dimensional
compact manifolds, not diffeomorphic to tori, which admit a
non-singular smooth circle action (hence Conjecture 6.1 of
\cite{Forni} does not hold).

\smallskip The goal of this paper is to prove that DUE examples do not
appear among non-toral \emph{homogeneous flows}, so that a non-toral
homogeneous flow always fails to be cohomology-free already because
the closure of its range has codimension higher than one. In fact, we
prove that for any homogeneous flow on a \emph{finite-volume}
homogeneous manifold $M$, except for the case of flows smoothly
isomorphic to linear toral flows, the closure of the range of the Lie
derivative operator on the space of smooth functions has countable
codimension, or, in other terms, the space of invariant distributions
for the flow has countable dimension. In particular, the Katok and
Greenfield-Wallach conjectures hold for general homogeneous flows on
finite-volume homogeneous manifolds.  Our main result can be stated as
follows.

\begin{theorem}
  \label{thm:Main}
  Let $G/D$ a connected finite volume homogeneous space. A homogeneous
  flow $(G/D, \flow_\R)$ is either smoothly isomorphic to a linear
  flow on a torus or it has countably many independent invariant
  distributions of bounded order (at most $\text{dim}(G/D) -1$ in the
  Sobolev sense, at most $1/2$ in the H\"older sense).
\end{theorem}

We can also prove an analogous theorem for affine diffeomorphisms.

\begin{theorem}
  \label{thm:Main_Affine}
  Let $G/D$ a connected finite volume homogeneous space.  An affine
  diffeomorphism $(G/D, \map)$ is either smoothly isomorphic to an
  ergodic translation on a torus or it has countably many independent
  invariant distributions of bounded order (at most $\text{dim}(G/D)$
  in the Sobolev sense, at most $1/2$ in the H\"older sense).
\end{theorem}

An important feature of our argument is that in the case of
\emph{partially hyperbolic} flows we prove the stronger and more
general result that any partially hyperbolic flow or diffeomorphism on
any compact manifold, not necessarily homogeneous, has infinitely many
distinct minimal sets (see Theorem~\ref{expandingminimal}). In
particular, we have a proof of the Katok and Greenfield-Wallach
conjectures in this case.  We are not able to generalize this result
to the finite-volume case. However, we can still prove that a
partially hyperbolic \emph{homogeneous} flow or an \emph{affine}
diffeomorphisms on a finite-volume manifold has countably many ergodic
probability measures (see Proposition~\ref{prop:5:2}).

In the non partially hyperbolic homogeneous case, that is, in the
quasi-unipotent case, by the Levi decomposition we are able to reduce
the problem to flows on semi-simple and solvable manifolds.  The
semi-simple case is reduced to the case of $\SL_2(\R)$ by an
application of the Jacobson--Morozov's Lemma which states that any
nilpotent element of a semi-simple Lie algebra can be embedded in an
$\sl_2(\R)$-triple. The solvable case can be reduced to the nilpotent
case for which our main result was already proved by the first two
authors in~\cite{FF}.  In both these cases the construction of
invariant distributions is based on the theory of unitary
representations for the relevant Lie group (Bargmann's classification
for $\SL(2,\R)$ and Kirillov's theory for nilpotent Lie groups).

The paper is organized as follows. In section~\ref{sec:2} we deal with
partially hyperbolic flows on compact manifolds.  In
section~\ref{sec:3} give the background on homogeneous flows that
allows us to reduce the analysis to the solvable and semi-simple
cases.  A further reduction is to consider quasi-unipotent flows
(sect.~\ref{sec:4}) and partially hyperbolic flows (sect.~\ref{sec:5})
on finite-volume non-compact manifolds; then the main theorem follows
easily (sect.~\ref{sec:6}).  Finally, in section~\ref{sec:7} we state
a general conjecture on the stability of homogeneous flows and a
couple of more general related open problems.

\smallskip {\bf Acknowledgments.}  The authors wish to thank A.~Katok
for several comments on the first version of this paper which led to a
significant improvement of the exposition and of the results.

L.~Flaminio was supported in part by the Labex~CEMPI
(ANR-11-LABX-07). G.~Forni was supported by NSF grant
DMS~1201534. F.~Rodriguez Hertz was supported by NSF grant
DMS~1201326.  L.~Flaminio would also like to thank the Department
Mathematics of the University of Maryland, College Park, for its
hospitality during the preparation of this paper.

%%% Local Variables: 
%%% mode: latex
%%% TeX-master: "InvDist_affine"
%%% End: 

\section{Partially hyperbolic flows and diffeomorphisms on compact
  manifolds}
\label{sec:2}

The goal of this section is to prove the following theorem.

\begin{theorem}
  \label{expandingminimal}
  Let $M$ be a compact connected manifold, $\flow_t$, $t\in\R$ or
  $t\in \Z$, an $\R$-action (a flow) or a $\Z$-action (a
  diffeomorphism) on $M$ and assume $\flow_t$ leaves invariant a
  foliation $\mathcal F$ with smooth leaves and continuous tangent
  bundle, e.g.\ the unstable foliation of a partially hyperbolic
  flow. Assume also that the action $\flow_t$ expands the norm of the
  vectors tangent to $\mathcal F$ uniformly. Then there are infinitely
  many different $\flow_t$-minimal sets.
\end{theorem}

The existence of at least one non trivial (i.e.\ different from the
whole manifold) minimal sets goes back G.~Margulis (see
\cite{Starkov}, \cite{MR96k:22022}, \cite{Kleinbock-Shah-Starkov})

and Dani \cite{MR794799}, \cite{MR870710}. A similar idea was already
used by R.~Ma\~n\'e in \cite{MR516217} and more recently by A.~Starkov
\cite{Starkov}, and F. \& J.~Rodriguez Hertz and R.~Ures \cite[Lemma
A.4.2 (Keep-away Lemma)]{MR2390288}, in different contexts.

\smallskip Theorem~\ref{expandingminimal} for diffeomorphisms can be
derived from the result for flows by passing to a suspension and for
flows it will follow almost immediately from the next lemma.

\begin{lemma}
  \label{induct}
  Let $\flow_t$ be a flow like in Theorem~\ref{expandingminimal}. For
  any $k$-tuple $p_1,\dots, p_k\in M$ of points in different orbits
  and for any open set $W\subset M$, there exist $\epsilon>0$ and
  $q\in W$ such that $d(\flow_t(q),p_i)\geq\epsilon$ for all $t\geq 0$
  and for all $i=1,\dots,k $.
\end{lemma}

Let us show how Theorem~\ref{expandingminimal} follows from
Lemma~\ref{induct}.

\begin{proof}[Proof of Theorem~\ref{expandingminimal}]
  Since $M$ is compact there is a minimal set $K$. Assume now by
  induction that there are $K_1,\dots K_k$ different minimal set, then
  we will show that there is a minimal set $K_{k+1}$ disjoint from the
  previous ones. Let $p_i\in K_i$, $i=1,\dots, k$ be $k$ points and
  take $q$ and $\epsilon>0$ from Lemma~\ref{induct}. Since
  $d(\flow_t(q),p_i)\geq\epsilon$ for any $t\geq 0$ and for
  $i=1,\dots, k$ we have that for any $i=1,\dots, k$,
  $p_i\notin\omega(q)$, the omega-limit set of $q$. Since the $K_i$'s
  are minimal, this implies that $K_i\cap \omega(q)=\emptyset$ for
  $i=1,\dots, k$. Take now a minimal subset of $\omega(q)$ and call it
  $K_{k+1}$.
\end{proof}

Let $F=T\mathcal F$ be the tangent bundle to the foliation with fiber
at $x\in M$ given by $F(x)=T_x\mathcal F(x)$. We denote by $d$ the
distance on $M$ induced by some Riemannian metric. Let $X$ be the
generator of the flow $\flow_t$.  Let also
$E(x)=(F(x)\oplus\<X(x)\>)^{\perp}$ be the orthogonal bundle and
$\mathcal E_r(x)=\exp_x(B^E_{r}(x))$ be the image of the $r$ ball in
$E(x)$ by the exponential map.  Let $f$ be the dimension of the
foliation $\mathcal F$ and $m$ the dimension of $M$. For $r \le r_0$
the disjoint union $\mathcal E:=\sqcup_{x\in M} \mathcal E_r(x)$ is a
$(m-f-1)$-dimensional continuous disc bundle over~$M$. Denote with
$d_{\mathcal F}$ and $d_{\mathcal E}$ the distances along the leaves
of $\mathcal F$ and $\mathcal E$, and let
\[
\mathcal F_r(x)=\{ y\in \mathcal F(x) \mid d_{\mathcal F}(y,x) \leq
r\} \subset \mathcal F(x)
\]
be the $f$-dimensional closed disc centered at $x$ and of radius
$r>0$.  Clearly $d\le d_{\mathcal F}$ and $d\le d_{\mathcal E}$.

We may assume that the Riemannian metric on $M$ is adapted so that
$\mathcal F_r(x) \subset\flow_{-t} \mathcal F_r(\flow_t x)$ for all
$x\in M$ and $r,t\ge 0$. In fact if $g$ is a Riemannian metric such
that $\| (\flow_{t})_* v \|_g \ge C \lambda^t \|v\|_g$ for all $v \in
F(x)$, all $x\in M$ and all $t\ge 0$, (where $\lambda >1$), then
setting $\hat g = \int^{T_0}_0 (\flow_t)^* g \, \hbox{d} t$, with
$T_0= - \log_\lambda (C/2)$, we have that, for all $v \in F(x)$ and
$x\in M$, the function $\| (\flow_{t})_* v\|_{\hat g}$ is strictly
increasing with $t$.

We may choose $r_1< r_0$ such that if $r \le r_1$ then, for all $x\in
M$, we have $d_{\mathcal F}(y,z) \le 2 d(y,z)$ for any $y,z\in
\mathcal F_r(x)$ and $d_{\mathcal E}(y,z) \le 2 d(y,z)$ for any
$y,z\in \mathcal E_r(x)$.

For $x\in M$, let
\[
V_{\delta,r}(x)=\bigcup_{z\in \mathcal E_{\delta}(x)}\mathcal F_r(z).
\]
There exists $ r_2 \le r_1$ such that, if~$r$ and~$\delta$ are both
less than~$r_2$, then $V_{\delta, 4 r}(x)$ is homeomorphic to a disc
of dimension $(m-1)$ transverse to the flow.

\smallskip\noindent{\bf Normalization assumption}: After a constant
rescaling of $X$ we may assume that given any $x\in M$,
$z,y\in\mathcal F(x)$ and $t\geq 1$ we have $d_{\mathcal
  F}(\flow_t(z),\flow_t(y))\geq 4\,d_{\mathcal F}(z,y)$. Henceforth,
in this section, we shall tacitly make this assumption.

\begin{proof}[Proof of Lemma~\ref{induct}]
  Let $p_1,\dots, p_k\in M$ be points belonging to different orbits
  and let $W\subset M$ be an open set.  We shall find $r>0$ and a
  point $x_0 \in W$ with $ \mathcal F_r(x_0)\subset W$ and then
  construct, by induction, a sequence of points $x_n\in M$ and of
  iterates $\tau_n\geq 1$ satisfying, for some $\delta >0$, the
  following conditions
  \begin{equation}
    \label{eq:1}\tag{$A_{n+1}$}
    \flow_{-\tau_n}(\mathcal F_r(x_{n+1}))\subset \mathcal
    F_r(x_n), \qquad  \text{\ for all\ } n\ge 0,
  \end{equation}
  and
  \begin{equation}
    \label{eq:2}\tag{$B_n$} 
    \flow_{T_n}x\in \mathcal
    F_r(x_n)\implies \flow_t(x)\notin\hbox{$\bigcup_i$}V_{\delta,
      2r}(p_i)\,, \quad \text{\ for all\ } t
    \in [0, T_{n+1}).
  \end{equation} 
  Here we have set $T_n:=\sum_{k=0}^{n-1}\tau_k$. Then defining $D_n:=
  \flow_{-T_n}\mathcal F_r(x_n)$ we have $D_{n+1}\subset D_n\subset
  \mathcal F_r(x_0)$ and any point $q\in\bigcap_n D_n\subset \mathcal
  F_r(x_0)\subset W$ will satisfy the statement of the Lemma.

  By the choice of an adapted metric we have
  \[
  d_{\mathcal F} (\flow_t (x),\flow_t(p)) \ge d_{\mathcal F}(x,p) \,,
  \qquad \text{ for all } p\in M \text{ and all } x\in \mathcal F(p)
  \,.
  \]
  This implies that for all $i=1, \dots ,k$ any $r>0$ and any $t\ge 0$
  \begin{equation}
    \label{eq:leaf_incl} 
    x\in \mathcal F_{4r}(p_i)\setminus\mathcal F_{2r}(p_i) \implies
    \flow_t(x)\not \in \mathcal F_{2r}(\flow_t(p_i)).
  \end{equation}
  Hence there exists $\delta_0 < r_2$ such that for all $\delta <
  \delta_0$ and all $r\le r_2$ we have:
  \begin{enumerate}[(i)]
  \item\label{item1} for all $i\in \{1,\dots ,k\}$,
    \[
    \flow_{[0,1]}\big( V_{\delta, 4r}(p_i) \setminus V_{\delta,
      2r}(p_i)\big) \cap V_{\delta, r}(p_i)= \emptyset \,.  \]
  \end{enumerate}
  The above assertion follows immediately by continuity if $p_i$ is
  not periodic of minimal period less or equal to $1$. If $p_i$ is
  periodic of period less or equal to $1$, then it follows by
  continuity from formula~\eqref{eq:leaf_incl}.
  
  As the orbits of $p_1, \dots, p_k$ are all distinct and the set $W$
  is open, we may choose a point $x_0\in W$ and positive real numbers
  $r,\delta< \delta_0$ so that the following conditions are also
  satisfied:
  \begin{enumerate}[(i)]
    \setcounter{enumi}{1}
  \item\label{item2} $\mathcal F_r(x_0)\subset W$;
  \item\label{item3} for all $i, j \in \{1,\dots ,k\}$, with $i\not=
    j$,
    \[
    \flow_{[0,1]} \big( V_{\delta, 4r}(p_i)\big) \cap \flow_{[0,1]}
    \big( V_{\delta, 4r}(p_j)\big)= \flow_{[0,1]} \big( V_{\delta,
      4r}(p_i)\big) \cap \bigcup_{t\in [0,1]} \mathcal F_r(\flow_tx_0)
    =\emptyset\,.
    \]
  \end{enumerate}

  If for all $t>0$ we have $ \mathcal F_r(\flow_t(x_0))\cap
  \hbox{$\bigcup_i$}V_{\delta, r}(p_i) =\emptyset$, then
  $d(\flow_t(x_0), p_i)> r$ for all $i=1,\dots k$ and all $t>0$,
  proving the Lemma with $q=x_0$ and $\epsilon =r$. Thus we may assume
  that
  \[
  \tau_0:=\inf\Big\{t>0\;:\;\mathcal F_r(\flow_t(x_0))\cap
  \hbox{$\bigcup_i$}V_{\delta, r}(p_i) \neq\emptyset\Big\} < \infty
  \]
  and define
  \[ \hat x_0 = \flow_{\tau_0}(x_0).
  \]
  The above condition~\eqref{item3} implies that $\tau_0\ge 1$, hence
  by the normalization assumption it follows that
  \begin{equation}
    \label{eq:3}
    \mathcal F_{5r}(\hat x_0) \subset  \flow_{\tau_0}\big(\mathcal F_r(x_0)\big).
  \end{equation}

  Assume, by induction, that points $x_k\in M$ and iterates
  $\tau_k\geq 1$ satisfying the conditions~({$A_n$}) and~\eqref{eq:2}
  have been constructed for all $k\in \{0,\dots, n\}$, and assume that
  the point $\hat x_n:= \flow_{\tau_n}(x_n)\in M$ is such that
  $\mathcal F_r(\hat x_n)$ intersects non-trivially some disc $
  V_{\delta,r}(p_i) $.  Since $V_{\delta, r}(p_i)$ is saturated by
  $\mathcal F$, it follows that $\mathcal F_{2r}(\hat x_n)\cap
  \mathcal E_{\delta}(p_i)$ consists of a unique point $z_n$ with
  $d_{\mathcal F}(z_n, \hat x_n) \le 2 r$; we define $x_{n+1}\in
  \mathcal F(\hat x_n)$ as the point at distance $3r$ on the geodesic
  ray in $\mathcal F(\hat x_n)$ going from $z_n$ to $ \hat x_n$ (or
  any point on the geodesic ray issued from $z_n$ if $\hat
  x_n=z_n$). Then we have
  \begin{equation}
    \label{eq:4}
    \mathcal F_ r(x_{n+1})\subset \mathcal F_{4r}(
    \hat x_n )\cap V_{\delta,4r}(p_i)\setminus
    V_{\delta,2r}(p_i)\,.
  \end{equation}
  Since $\mathcal F_ r(x_{n+1})\subset V_{\delta, 4r}(p_i)$, for all
  $t\in (0,1)$ we have
  \[
  \mathcal F_r(\flow_t x_{n+1}) \subset \flow_t \mathcal F_r( x_{n+1})
  \subset \flow_{[0,1]}\big( V_{\delta,4r}(p_i) \setminus V_{\delta,
    2r}(p_i)\big)\,.
  \]
  By the disjointness conditions~\eqref{item1} and~\eqref{item3}, it
  follows that, for all $t\in (0, 1]$,
  \[\mathcal F_r(\flow_t x_{n+1})\cap \bigcup_{i=1}^k V_{\delta,
    r}(p_i) =\emptyset\,.\]

  It follows that if we define
  \[
  \tau_{n+1}:=\inf \Big\{t>0\;:\;\mathcal F_r(\flow_t(x_{n+1}))\cap
  \hbox{$\bigcup_iV_{\delta, r}(p_i)$}\neq\emptyset\Big\},\qquad \hat
  x_{n+1}= \flow_{\tau_{n+1}}(x_{n+1})
  \]
  (assuming $\tau_{n+1}<+\infty$), then $\tau_{n+1}\ge 1$, and by the
  normalization assumption and by the inclusion in
  formula~\eqref{eq:4} we have
  \[
  \mathcal F_{r}(x_{n+1}) \subset \mathcal F_{4r}(\hat x_{n})=
  \mathcal F_{4r}(\flow_{\tau_{n}} x_{n}) \subset
  \flow_{\tau_{n}}\big(\mathcal F_r(x_{n})\big)
  \]
  and by construction, having set $T_{n+2}:=\sum_{k=0}^{n+1}\tau_k$,
  we also have
  \[
  x\in D_{n+1}:=\flow_{-T_{n+1}}(\mathcal F_r(x_{n+1}))\implies
  \flow_t(x)\notin\hbox{$\bigcup_i$}V_{\delta, 2r}(p_i)\,, \quad
  \text{ for all } t \in [0, T_{n+2}).
  \]
  The inductive construction is thus completed.  As we explained above
  we have that~$(D_n)$ is a decreasing sequence of closed
  sub-intervals of $\mathcal F_r(x_0)$ and that any point $q\in
  \bigcap_n D_n$ satisfies $\flow_t(q) \notin \bigcup_i V_{\delta,
    r}(p_i)$ for all $t\ge 0$.

  The above inductive construction may fail if at some stage $n\geq 0
  $ we have $\tau_n = +\infty$. In this case let $q$ be any point in
  $\flow_{-T_n}(\mathcal F_r(x_n))$. Again such a point $q\in W$
  satisfies the statement of the Lemma, hence the proof is completed.
\end{proof}

\begin{corollary}
  \label{pcf_maps}
  Let $\map$ be a partially hyperbolic $C^1$ diffeomorphism on a
  compact connected manifold $M$. If the map $\map$ satisfies the
  accessibility condition, there exists $\alpha>0$ such that every
  H\"older invariant distribution of order at most $\alpha$ belongs to
  the closure (in the space $\mathcal D'(M)$ of all distributions on
  $M$) of the linear space spanned by an invariant measure and by the
  family of Periodic Cycle Functionals. In particular, the space
  spanned by all Periodic Cycle Functionals is infinite dimensional.
\end{corollary}

\begin{proof}
  By~\cite{Wilkinson} results there exists $\alpha$, $\beta \in (0,1)$
  with $\beta \geq\alpha$ such that a function $f \in C^\beta(M)$
  belongs to the joint kernel of all Periodic Cycle Functionals
  (PCF's) and of a $\map$-invariant measure $m$ if and only if $f$ is
  a coboundary with a continuous, hence $C^\alpha$, primitive (i.e.\
  transfer function).  It follows that $f$ belongs to the kernel of
  all H\"older invariant distributions of order at most $\alpha>0$.
  Since the space $C^\infty(M)$ is reflexive, by the Hahn-Banach
  theorem if a distribution $D \in \mathcal D'(M)$ does not belong to
  the closure in $\mathcal D'(M)$ of the linear space spanned by the
  invariant measure $m$ and by the family of all PCF's, then there
  exists a function $f\in C^\infty(M)$ such that $f$ has zero average
  with respect to $m$ and belongs to the kernel of all PCF's but
  $D(f)\not=0$.  However, $f$ is a coboundary with $C^\alpha$ transfer
  function, which implies that $D(f)=0$. This contradiction implies
  that the space $\mathcal D^\alpha_\map(M)$ of all H\"older invariant
  distributions of order at most $\alpha$ is a subset of the closure
  in the topology of $\mathcal D'(M)$ of the linear space spanned by
  the invariant measure $m$ and by the family of Periodic Cycle
  Functionals. Since by Theorem~\ref{expandingminimal} the partially
  hyperbolic diffeomorphism $\map$ has infinitely many distinct
  minimal sets (it follows by considering a suspension flow), the
  space $\mathcal D^\alpha_\map(M)$ is infinite dimensional for any
  $\alpha\geq 0$, hence the argument is concluded.
\end{proof}
  
In order to prove an analogous result for partially hyperbolic flows,
we extend Wilkinson's theorem on solutions of the cohomological
equation to the case of flows. As we shall see, it is a simple
corollary of the theorem for maps.
  
\begin{theorem}
  \label{cohomeq_PHflows}
  Let $\flow_t$ be a $C^1$ partially hyperbolic flow generated by a
  vector field $X$ on a compact manifold $M$. If the flow $\flow_t$
  satisfies the accessibility conditions, then for any $\beta\in
  (0,1)$ there exists $\alpha \in (0,1)$ such that any H\"older
  function $f \in C^\beta(M)$ which belongs to the joint kernel of all
  Periodic Cycle Functionals is cohomologous to a constant over the
  flow $\flow_t$ with a H\"older transfer function in $C^\alpha(M)$,
  that is, there exists a H\"older (transfer) function $u\in
  C^\alpha(M)$ and a constant $c\in \C$ such that the following
  cohomological equation holds (in the distributional sense)
  \[
  Xu = f-c \,.
  \]
  Conversely, if for any function $f\in C^\beta(M)$ there exists a
  continuous function $u\in C^0(M)$ and a constant $c\in \C$ such that
  the above cohomological equation holds, then $u\in C^\alpha(M)$,
  hence $f$ belongs to the kernel of all Periodic Cycle Functionals.
\end{theorem}
\begin{proof}
  For any $t>0$, let $\map_{(t)}$ denote the time-$t$ map of the flow
  $\flow_t$. The $C^1$ diffeomorphism $\map_{(t)}$ is partially
  hyperbolic and its stable and unstable foliations coincide with the
  stable and unstable foliations of the flow. It follows that
  $\map_{(t)}$ has the accessibility property and that the set of its
  stable-unstable paths coincides with the set of stable-unstable
  paths for the flow. Let $f\in C^\beta(M)$ belong to the kernel of
  all Periodic Cycle Functionals (PCF's) for the flow. By the above
  remark and by the definition of the PCF's, it follows that for any
  $t >0$ the function
  \[
  f_t := \int_0^t f \circ \flow_s\,\D s
  \]
  belongs to the kernel of all PCF's for the time-$t$ map
  $\map_{(t)}$.  Let $m$ be any invariant measure for the flow
  $\flow_t$, hence for all its time-$t$ maps. By Wilkinson's theorem
  \cite{Wilkinson}, there exists a unique function $u_t \in C^0(M)$,
  of zero average with respect to $m$, and a constant $c_t\in \C$ such
  that
  \[
  f_t -c_t = u_t \circ \map_{(t)} -u_t\,.
  \]
  We claim that for all $t>0$ we have $u_{2t}= u_t$ and $c_{2t}=
  2c_t$. In fact,
  \[
  \begin{aligned}
    f_{2t} &= f_t + f_t \circ \map_{(t)} = u_t \circ \map_{(t)} -u_t  +c_t + (u_t \circ \map_{(t)} -u_t +c_t )\circ \map_{(t)}     \\
    &= u_t \circ\map_{(t)}-u_t + (u_t \circ \map_{(t)}-u_t)\circ
    \map_{(t)} + 2c_t= u_t \circ \map_{(2t)}-u_t + 2 c_t \,.
  \end{aligned}
  \]
  It follows by the uniqueness of the solution that $u_{2t}= u_t$ and
  $c_{2t}= 2c_t$ as claimed.  By the above claim it follows that for
  all $n\geq 0$ we have
  \[
  u_{1/2^n} = u_1 \quad \text{ and } \quad c_{1/2^n} = c_1 / 2^n \,.
  \]
  We can therefore write, after multiplying on both sides by the
  factor $2^n$,
  \[
  2^n \int_0^{1/2^n} f \circ \flow_s \,\D s \, - \, c_1 = 2^n (u_1
  \circ \map_{(1/2^n)} - u_1) \,.
  \]
  By taking the limit as $n\to +\infty$ (in the sense of
  distributions), we finally derive the equation
  \[
  f - c_1 = X u_1 \,.
  \]
  We have thus proved that the cohomological equation for the flow has
  a solution $(u,c) =(u_1, c_1)$ with a continuous transfer function
  $u\in C^0(M)$. In order to prove that the transfer function is in
  fact H\"older, we argue that under the assumption that the function
  $f \in C^\beta (M)$, the transfer function $u\in C^\alpha(M)$. In
  fact, the above cohomological equation for the flow implies, by
  integration along the flow up to time $t=1$, the following
  cohomological equation for the time-$1$ map:
  \[
  f_1 - c = u \circ \map_{(1)} -u\,.
  \]
  It follows then by Wilkinson's theorem~\cite{Wilkinson}, since the
  time-$1$ map is partially hyperbolic and satisfies the accessibility
  condition, the function $f\in C^\beta(M)$, hence the integrated
  function $f_1\in C^\beta(M)$ as well, and the transfer function
  $u\in C^0(M)$, that in fact $u \in C^\alpha(M)$.  Finally, we note
  that if the above cohomological equation for the flow has a H\"older
  solution, then the function $f\in C^\beta(M)$ belongs to the kernel
  of all PCF's (for the flow) since all PCF's are invariant
  functionals, bounded on any H\"older space, which vanish on constant
  functions. The argument is therefore completed.
\end{proof}
From the above theorem and from Theorem~\ref{expandingminimal} we can
then derive the following corollary, whose proof is entirely analogous
to that of Corollary~\ref{pcf_maps}.
\begin{corollary}
  \label{pcf_flows}
  Let $\flow_t$ be a partially hyperbolic $C^1$ flow on a compact
  connected manifold $M$. If the flow $\flow_t$ satisfies the
  accessibility condition, there exists $\alpha>0$ such that every
  H\"older invariant distribution of order at most $\alpha$ belongs to
  the closure (in the space $\mathcal D'(M)$ of all distributions on
  $M$) of the linear space spanned by an invariant measure and by the
  family of all Periodic Cycle Functionals. In particular, the space
  spanned by all Periodic Cycle Functionals is infinite dimensional.
\end{corollary}

%%% Local Variables: 
%%% mode: latex
%%% TeX-master: "InvDist_affine"
%%% End: 

\section{Homogeneous flows and affine diffeomorphisms}
\label{sec:3}

Henceforth $G$ will be a connected Lie group and $G/D$ a finite volume
space; this means that $D$ is a closed subgroup of $G$ and that $G/D$
has a finite $G$-invariant (smooth) measure. The group $D$ is called
the \emph{isotropy group} of the space $M=G/D$. As we are only
interested in the quotient space $M$, we may assume that $G$ is simply
connected and that the isotropy group $D$ is a \emph{quasi-lattice},
i.e.\ that the largest connected normal subgroup of $D$ is reduced to
the identity.

Let $\mathfrak g$ be the Lie algebra of $G$. The exponential map $\exp
\colon \mathfrak g \to G$ sets up a bijective correspondence between
elements $X\in \mathfrak g$ and one-parameter subgroups $(\flow_t=\exp
tX)_{t\in \R}$. We denote a one-parameter subgroup $(\flow_t)_{t\in
  \R}$ of $G$ by $\flow_\R$. The flow generated by left translations
by this one-parameter subgroup on the finite volume space $G/D$ will
be denoted $(G/D, \flow_\R)$ or simply $\flow_\R$.

If $G$ is simply connected, then the group $\Aut(G)$ of Lie group
automorphisms of~$G$ is identified with the algebraic group
$\Aut(\mathfrak g)$ of Lie algebra automorphisms of~$\mathfrak g$, via
the map associating to an automorphism its differential at the
identity. In this case we shall not make a distinction between these
groups.

For any group $G$ we denote its center by $Z(G)$. For any element $x$
of a group $G$ we let $\operatorname{Int}(x)\in \Aut(G)$ be the
\emph{inner automorphism} given by the conjugation by $x$. For any
normal subgroup $H\vartriangleleft G$, we denote by
$\operatorname{Int}_H(x)$ the automorphism of $H$ given by the
conjugation by $x\in G$. The center of $Z(G)$ is precisely the kernel
of the map $\operatorname{Int} \colon G\to \Aut(G)$. The \emph{adjoint
  representation} $\Ad\colon G \to \Aut(\mathfrak g)$ is defined by
letting $\Ad(g)$ be the differential of the inner automorphism $
\operatorname{Int}(g)$ at the identity of~$G$.

An \emph{affine map} of a Lie group $G$ is the composition of a
continuous automorphism~$A$ of~$G$ and a (left) translation by an
element of~$G$. We denote by $\map= u A$ the affine map defined by
$\map(x)=uA(x)$ for all $x\in G$, where $u\in G$ and $A\in
\Aut(G)$. Affine maps form a group under composition which may be be
identified to the semi-direct product $\Aff(G):=\Aut(G)\rtimes G$. As
$G$ is a normal subgroup of $\Aff(G)$ conjugation by the affine map
$\map=u A$ yields an automorphism of $G$, which is easily seen to be
given by $\operatorname{Int}_G(\map):=\operatorname {Int}(u)\circ A$.

An affine map $\map= uA$ induces a smooth quotient map of $G/D$ if and
only if $A(D)\subset D$ and it induces a diffeomorphism of $G/D$ if
and only if the equality $A(D)=D$ holds true. We call \emph{affine
  diffeomorphism of $G/D$} a diffeomorphism of $G/D$ induced by an
affine map of $G$ and denote by $\Aff(G/D)$ the group of such
diffeomorphisms. The group $\Aff(G/D)$ of affine diffeomorphism of
$G/D$ is a quotient group of the group $\Aff(G)$ of affine maps
of~$G$.

\begin{lemma}
  \label{lem:homogflow3:1}
  Let $G/D$ be a finite volume space with $D$ a quasi-lattice. Let
  $\map=uA$ be an affine map of $G$ projecting to an affine
  diffeomorphism $\bar \map$ of $G/D$ and let $\flow_\R$ a
  one-parameter subgroup of $G$. Then
  \begin{itemize}
  \item The flow by left translations by $\flow_\R$ on $G/D$ commutes
    with the map $\bar \map$ if and only if the subgroup $\flow_\R$ is
    fixed by the automorphism $\operatorname{Int}(u)\circ A\in
    \Aut(G)$.
  \item The affine diffeomorphism $\bar \map$ is the identity on $G/D$
    if and only if the map~$\map$ is the right translation by an
    element in $ D$, i.e.\ if and only if $u=\gamma$ and
    $A=\operatorname{Int}(\gamma^{-1})$, for $\gamma\in D$. Thus the
    map
    \[
    \bar \map \in \Aff(G/D) \to \operatorname{Int}_G(
    \map)=\operatorname{Int}(u)\circ A \in \Aut(G)
    \]
    is well defined group homomorphism, as the right hand side does
    not depend on the choice of the affine map $\map=uA\in \Aff(G)$
    projecting to $\bar \map$.
  \end{itemize}
\end{lemma}
\begin{proof}
  We have $\flow_t u A(x) D =u A( \flow_t x)D$ for all $x\in G$ and
  all $t\in \R$, if and only if $A( \flow_{-t}) u^{-1}\flow_t u \in y
  D y^{-1}$ for all $y\in G$ and all $t\in \R$. But $D$ is a
  quasi-lattice of $G$, hence it does not have any non-trivial
  connected normal subgroups. The first statement of the lemma
  follows.

  The map $\bar \map$ is the identity on $G/D$ if and only if it
  commutes with every one-parameter subgroup of $G$. By the previous
  statement, this condition is equivalent to the identity
  $A=\operatorname{Int}(u^{-1})$. Thus $xD=\bar \map (xD) = x u D$ for
  all $x\in G$. This implies that $u\in D$.
\end{proof}

The conclusion of the above lemma may be stated by saying that, if $D$
is a quasi-lattice, the group $\Aff(G/D)$ of affine diffeomorphism of
$G/D$ is isomorphic to the ``adjoint group''
$\{\operatorname{Int}_G(\map)\in \Aut(G)\mid \map \in
\Aff(G)\}$. Hence we obtain

\begin{corollary}
  \label{cor:3homogflow:1}
  Let $G/D$ be a finite volume space with $D$ a quasi-lattice.  Any
  finite subgroup of affine map of the homogenous space $G/D$, acts as
  a finite group of automorphims on the Lie algebra of $G$.
\end{corollary}

Our results for affine diffeomorphisms will be derive from the
corresponding results for flows by the following method. By a
modification of a construction of S.~G.~Dani~\cite{MR0444835} it is
possible to reduce the general case of affine diffeomorphisms to the
case of homogeneous flows up to the action of a finite subgroup.

Invariant distributions for affine maps are then obtained from
invariant distributions for homogenous flows. This reduction is based
on the fundamental well-known fact that the spaces of invariant
distributions for a diffeomorphism and for its suspension flow are
isomorphic. We include a proof below for the convenience of the
reader.

\begin{definition}
  \label{def:suspension}
  The \emph{suspension flow} of a diffeomorphism $\map$ on a smooth
  manifold $M$ is a flow $\{ \flow_t\}$ on a manifold $ \susp{M}$ such
  that $M \subset \susp{M}$ is transverse to the flow and the
  map~$\map$ coincides with the first return map to $M$ and with the
  time-$1$ map of the flow $\{ \flow_t\}$ on $ \susp{M}$ (in other
  terms the return time function is constant and equal to $1$).
\end{definition}

The suspension flow is unique up to diffeomorphism and can be
constructed as follows.  Let $ \susp{M}$ denote the quotient of the
product $M\times \R$ with respect to the equivalence relation defined
as
\[
(x,r) \sim_\map (\map^{-1}(x), r+1) \,, \quad \text{ for all } \,
(x,r) \in M\times \R\,.
\]
Let $\{ \flow_t\}$ be the projection to $ \susp{M}= M\times \R/
{\sim}_\map $ of the `vertical' flow on $M\times \R/{\sim}_\map$, that
is, of the flow $\{ \flow_t \}$ defined as
\[
\flow_t (x, r) = (x, r+t) \,, \quad \text{ for all } \, t\in \R\,
\text{ and } \,(x,r) \in M\times \R\,.
\]
The flow $\{ \flow_t\}$ is generated by the vector field $ V$ on
$\susp{M}$ obtained by projection of the `vertical' vector field
$(0,\partial/\partial r)$ on $M\times \R$.  The manifold $M$ is
diffeomorphic to the manifold $M\times \{0\} / {\sim}_\map$, by the
composition of the inclusion $M \to M\times \{0\}$ and of the
projection $M\times \R\to \susp{M}$, and the flow $\{ \flow_t\}$ is
the suspension flow of the diffeomeorphism $f$ as a map on $M\times
\{0\} / {\sim}_\map$.

\begin{proposition}
  \label{prop:3homogflow:1}
  Let $(\susp{M}, \flow_t)$ be the suspension flow of a diffeomorphism
  $\map$ of a smooth manifold $M$. Then the space of $\map$-invariant
  distributions on $M$ of a given Sobolev or H\"older order (resp.\ of
  $\map$-invariant measures on $M$) is isomorphic to the space of
  $\flow_t$-invariant distributions on $\susp{M}$ of the same Sobolev
  or H\"older order (resp.\ of $\flow_t$-invariant measures on
  $\susp{M}$).
\end{proposition}

The statement about measures in the above Proposition follows
immediately from the above construction of the suspension flow. The
complete proof of the Proposition will follow from the
Lemma~\ref{lem:3homogflow:1} below.

Let $M$ be a manifold and let $ \susp{M}$ be the suspension space as
defined above. For any function $\theta\in C_0^\infty(-1/2, 1/2)$,
defined on $\R$, such that
\begin{equation}
  \label{eq:cut-off}
  \int_{-1/2}^{1/2} \theta (r) dr = 1\,,
\end{equation}
we define a continuous linear operator $ E_\theta: C_0^\infty(M) \to
C_0^\infty( \susp{M})$ as follows.  For any $f\in C_0^\infty(M)$, let
\begin{equation}
  \label{eq:Eop}
  \widetilde E_\theta (f) (x,r) = \sum_{n\in \Z}  f \circ \map^{-n} (x) \theta(r+n) \,, \quad \text{\ for all\ } (x,r)\in 
  M\times \R\,.
\end{equation}
Since $\theta\in C_0(-1/2,1/2)$ the above sum is finitely supported
and defines a function $\widetilde E_\theta (f)\in
C^\infty(M\times\R)$, which is constant on all equivalence classes of
the relation~$\sim_\map$, hence descends to a function $E_\theta(f)
\in C^\infty(\susp{M})$.  It is immediate from the definition that the
function $E_\theta(f) $ has compact support and that
\[
E_\theta: C_0^\infty(M) \to C_0^\infty(\susp{M})
\]
is a continuous linear operator.  By the above definition, it is also
clear that $E_\theta$ extends to a linear continuous operator from
Sobolev and H\"older spaces of functions on $M$ to Sobolev,
respectively H\"older, spaces of functions on $\susp{M}$ of the same
order.

Let $E^\ast_\theta : \mathcal D'(\susp{M}) \to \mathcal D'(M)$ denote
the dual operator on distributions, which is defined as follows:
\[
E^\ast_\theta (D) (f) = D (E_\theta(f)) \,, \quad \text{ for all }
f\in C_0^\infty(M)\,.
\]
Since the operator $E_\theta$ extends to a linear continuous operator
on Sobolev and H\"older spaces of functions and it preserves the
order, its dual $E^\ast_\theta$ maps Sobolev and H\"older spaces of
distributions on $\susp{M}$ to Sobolev, respectively H\"older, spaces
of distributions on $M$ of the same order.

\begin{lemma}
  \label{lem:3homogflow:1}
  The restriction of the operator $E^\ast_\theta$ to the subspace
  $\mathcal I_V(\susp{M})$ of distribution in $\subset \mathcal
  D'(\susp{M})$ invariant under the suspension flow $\{ \flow_t\}$ is
  a continuous linear operator $E^\ast$ which does not depend on the
  choice of the function $\theta\in C^\infty_0(-1/2,1/2)$ satisfying
  the conditions in formula~\eqref{eq:cut-off}. Moreover, the operator
  $E^\ast: \mathcal D'(\susp{M}) \to \mathcal D'(M)$ induces an
  isomorphism between the space $\mathcal I_V(\susp{M})$ and the
  subspace $\mathcal I_\map(M)$ of distributions invariant under the
  diffeomorphism $\map$ on $M$, which preserves the Sobolev as well as
  the H\"older order of invariant distributions.

\end{lemma}
\begin{proof}
  Let $\theta_1 $, $\theta_2\in C_0^\infty(-1/2,1/2)$ be any two
  functions such that
  \[
  \int_{-1/2}^{1/2} \theta_1 (r) dr = \int_{-1/2}^{1/2} \theta_2 (r)
  dr = 1 \,.
  \]
  We claim that, for any $f\in C_0^\infty(M)$, the function
  $E_{\theta_1} (f) - E_{\theta_2}(f)$ is a smooth coboundary for the
  suspension flow. In fact, since $\theta_1-\theta_2$ has zero average
  on $(-1/2, 1/2)$ there exists a smooth function $\chi\in
  C_0^\infty(-1/2, 1/2)$ such that
  \[
  \theta_1-\theta_2 = \frac{d\chi}{dr}\,.
  \]
  Let $\widetilde F \in C^\infty(M\times \R)$ be the function defined
  as
  \[
  \widetilde F(x,r):=\sum_{n\in \Z} f \circ \map^{-n} (x) \chi(r+n)
  \,, \quad \text{ for all } (x,r)\in M\times \R\,.
  \]
  The function $\widetilde F$ is well-defined and it projects to a
  smooth function $F\in C_0^\infty( \susp{M})$.  In addition, from the
  identity
  \[
  \widetilde E_{\theta_1} (f) (x,r) - \widetilde E_{\theta_2} (f)(x,r)
  = \frac{d \widetilde F}{dr} (x,r) \,, \quad \text{ for all }
  (x,r)\in M\times \R\,,
  \]
  it follows by projection that $E_{\theta_1} (f)- E_{\theta_2}(f) = V
  F$, as claimed.  For any distribution $D\in \mathcal D'(\susp{M})$
  invariant under the suspension flow and for all $f\in C^{\infty}(M)$
  we then have
  \[
  \left( E^\ast_{\theta_1}( D) - E^\ast_{\theta_2}(D)\right) (f) = D
  \left (E_{\theta_1} (f) - E_{\theta_2} (f)\right) = D (V F) =0\,.
  \]

  We have thus proved that the restriction of $E^\ast_\theta$ to the
  subspace of invariant distributions for the suspensionf flow is a
  continuous linear operator $E^\ast$ independent of the choice of the
  function $\theta \in C_0^\infty(-1/2, 1/2)$ with integral equal to
  $1$.

  \smallskip Next we prove that the linear operator $E^\ast_\theta$
  maps the subspace $\mathcal I'_V(\susp{M})$ of invariant
  distributions for the suspension flow $\flow_t$ into the subspace
  $\mathcal I'_\map(M)$ of invariant distributions for the
  diffeomorphism~$\map$. By construction, for any $f \in
  C_0^\infty(M)$ we have the identity
  \[
  \widetilde E_\theta(f \circ \map) (x,r) = \widetilde E_\theta(f )
  (x,r+1) \,, \quad \text{ for all } (x,r)\in M\times \R\,,
  \]
  which, under projection on $\susp{M}$, implies the following
  identity:
  \[
  E_\theta(f \circ \map) = E_\theta(f) \circ \flow_1 \,.
  \]
  It follows that, for all $D\in \mathcal I'_V(\susp{M})$ and all
  $f\in C_0^\infty(M)$, we have
  \[
  E^\ast_\theta (D) (f\circ \map) = D\left( E_\theta(f \circ
    \map)\right) = D\left( E_\theta(f) \circ \flow_1\right) = D\left(
    E_\theta(f)\right) = E^\ast_\theta (D) (f)\,.
  \]
  Hence $E^\ast_\theta ( D) \in \mathcal I'_\map(M)$.

  \smallskip That the continuous linear operator $E_\theta^\ast$ is an
  isomorphism of $\mathcal I'_V(\susp{M})$ onto $ \mathcal I'_\map(M)$
  follows from the construction of an inverse operator. Let $I:
  C_0^{\infty}(\susp{M}) \to C_0^{\infty}(M)$ the continuous linear
  operator defined as follows. For any $F\in C_0^\infty(\susp{M})$,
  the function $I(F)\in C_0^\infty(M)$ is defined as
  \begin{equation}
    \label{eq:Iop}
    I(F) (x) = \int_{-1/2}^{1/2}   F\circ \flow_t (x,0)\,\D t \,, 
    \quad \text{\ for all\ } x\in M\,. 
  \end{equation}
  By the above definition, it is clear that the operator $I$ extends
  to a linear continuous operator from Sobolev and H\"older spaces of
  functions on $\susp{M}$ to Sobolev, respectively H\"older, spaces of
  functions on $M$ of the same order.

  It is immediate from the construction that
  \[
  (I \circ E_\theta) (f) =f \,, \quad \text{ for all } f \in
  C_0^\infty(M),
  \]
  hence the dual operator $I^\ast : \mathcal D'(M) \to \mathcal
  D'(\susp{M})$ is a right inverse of the operator $E^\ast_\theta$:
  \[
  (I^\ast \circ E^\ast _\theta) (D) = D, \quad \text{ for all } D \in
  \mathcal D' (\susp{M})\,.
  \]
  We claim that for all $F\in C_0^\infty(\susp{M})$, the function
  $(E_\theta \circ I)(F) -F$ is a smooth coboundary for the suspension
  flow. Let $\widetilde F \in C^{\infty}(M\times\R)$ be any lift of
  the function~$F$ on $\susp{M}$ to $M\times \R$. By construction
  \[
  \int_0^1 [(\widetilde E_\theta \circ I)(F) -\widetilde F] (x,n+t) \,
  \D t =0 \,, \quad \text{ for all }\,(x,n)\in M\times \Z\,.
  \]
  It follows that the smooth function $\widetilde U$ on $M\times \R$
  defined as
  \[
  \widetilde U (x,r) = \int_0^r [(\widetilde E_\theta \circ I)(F)
  -\widetilde F] \circ \flow_t (x,0)\, \D t \,, \quad \text{ for all
  }\,x\in M,
  \]
  has a well-defined projection $U\in C_0^\infty(\susp{M})$ to the
  quotient $\susp{M}= M\times \R/ {\sim}$, since the function
  $(\widetilde E_\theta \circ I)(F) -\widetilde F$ has a well-defined
  projection and
  \[
  \widetilde U (x,n) =0 \,, \quad \text{ for all }\,(x,n)\in M\times
  \Z\,.
  \]
  In addition, again by construction we have
  \[
  \frac{d}{dr} \widetilde U (x,r) = [(\widetilde E_\theta \circ I)(F)
  -\widetilde F] (x,r)\,, \quad \text{ for all } (x,r)\in M\times
  \R\,, \] hence $VU = (E_\theta \circ I)(F) -F$, that is, the
  function $(E_\theta \circ I)(F) -F$ is a coboundary for the
  suspension flow, as claimed.  It follows that for any invariant
  distribution $D \in \mathcal I_V (\susp{M})$ we have
  \[
  D [(E_\theta \circ I)(F)] = D (F)\,, \quad \text{ for all } F\in
  C_0^\infty (\susp{M}) \,,
  \]
  hence the dual operator $I^\ast : \mathcal I_\map(M) \to \mathcal
  I_V(\susp{M})$ is a left inverse of the operator $E^\ast_\theta$,
  that is,
  \[
  (I^\ast \circ E^\ast _\theta) (D) = D, \quad \text{ for all } D \in
  \mathcal I_V (\susp{M})\,.
  \]
  We therefore conclude that the continuous linear operator $E^\ast:
  \mathcal I_V(\susp{M}) \to \mathcal I_\map(M)$ is an isomorphism as
  its inverse $I^\ast: \mathcal I_\map(M) \to \mathcal I_V(\susp{M})$
  is well-defined and continuous. Since the operator $I$ extends to a
  linear continuous operator from Sobolev and H\"older spaces of
  functions on $\susp{M}$ to Sobolev, respectively H\"older, spaces of
  functions on $M$ of the same order, its dual $I^\ast$ maps Sobolev
  and H\"older spaces of distributions on $M$ to Sobolev, respectively
  H\"older, spaces of distributions on $\susp{M}$ of the same
  order. It follows that the isomorphism $E^\ast: \mathcal
  I_V(\susp{M}) \to \mathcal I_\map(M)$ preserves the Sobolev as well
  as the H\"older order of invariant distributions.
\end{proof}

When a group $\Hgroup$ operates on a group $\Ggroup$ by automorphisms
of $\Ggroup$, we denote by $\Aff_{\Hgroup}(\Ggroup)$ the group of
affine maps $\map= uA$ of $\Ggroup$ such that $A\in \Hgroup$. We also
denote by $\Aff_{\mathcal H}(\Ggroup/ \Dlattice)$ the group of
diffeomorphisms of $\Ggroup/\Dlattice$ induced by affine maps
in~$\Aff_{\Hgroup}(\Ggroup)$.

\begin{lemma}
  \label{lemma:reduct_transl}
  Let $\Ggroup$ and $\Hgroup$ be linear algebraic groups and assume
  that $\Hgroup$ acts on~$\Ggroup$ by rational automorphisms. Let $
  \Gsubgp \subset \Ggroup$ be an $\Hgroup$-invariant normal subgroup
  of $\Ggroup$ and let $ \Dlattice$ be a quasi-lattice in
  $\Gsubgp$. Let $ \map \in \Aff_{\Hgroup}(\Gsubgp/\Dlattice)$ be an
  affine diffeomorphism.

  There exist a connected Lie group $\Gphi $ such that $\Gsubgp <
  \Gphi$, a quasi-lattice $\biglattice$ of~$\Gphi$ containing
  $\Dlattice$, a non trivial one-parameter subgroup $\flow_\R \subset
  \Gphi$ and an affine map $F$ of $\Gphi$ with the following
  properties:
  \begin{enumerate}
  \item The map $F$ is induces a periodic affine map $F_m$ of
    $\Gphi/\biglattice$. Let $\finitegp$ be the group generated by
    $F_m$.
  \item the map $F$ commutes with the action of $\flow_\R$ by left
    translations on $\Gphi$; hence we obtain a quotient flow
    $(\finitegp\backslash \Gphi/ \biglattice, \phi_\R)$ on the double
    coset space $\finitegp\backslash \Gphi/\biglattice$.
  \item the flow $(\finitegp\backslash \Gphi/ \biglattice, \phi_\R)$
    is smoothly conjugate to a suspension of the affine map $\phi$ on
    $\Gsubgp/\Dlattice$.
  \end{enumerate}
  The map $F$ yields an affine map of the Levi factor $\mathcal L$ of
  $\Gphi$.
\end{lemma}
\begin{proof}
  The argument is adapted from the initial step the proof of
  Theorem~7.1 in Dani's paper~\cite{MR0444835}.

  Since $\Aff_\Hgroup(\Gsubgp/\Dlattice) < \Aff_\Hgroup(\Ggroup)$, we
  may consider the algebraic hull $\<\map\>^Z$ of the subgroup
  generated by $\map$ in the algebraic group $\Aff_\Hgroup(\Ggroup)$.
  The group $\<\map\>^Z$ is Abelian with finitely many connected
  components, hence there exists an element $\hatf \in\<\map\>^Z $ of
  order $m$ and a one-parameter group $
  \bar{\flow}_\R=(\bar{\flow}_t)_{t\in \R} < \<\map\>_0^Z$ commuting
  with $\hatf$ such that $\map= \hatf \circ \bar{\flow}_1$. This
  implies that $ \map^m= \bar{\flow}_{m}$. (Here we used the symbol
  $\circ$ for the product in $\Aff_\Hgroup(\Ggroup)$ as we think of
  it as a group of maps of $\mathcal A$.)

  The group $\bar{\flow}_\R$ is non-trivial, otherwise $\map =\hatf$
  and $\map^m= \text{id}$, but by hypothesis the order of $\map$ is
  infinite.

  \smallskip Let us embed the group $\Ggroup$ (and its subgroups) into
  $\text{\rm Aff}_{\Hgroup} (\Ggroup)$ by identifying elements of the
  group with the associated left translations.  Since the subgroup
  $\Gsubgp$ is normal in $\Ggroup$ and is $\mathcal H$ invariant, it
  follows that it embeds as a normal subgroup of $\text{\rm
    Aff}_{\Hgroup} (\Ggroup)$.  The subgroup $\bar {\Gsubgp}$ of ${\rm
    Aff}_{\Hgroup} (\Ggroup)$ generated by the subgroups
  $\bar{\phi}_\R$ and $ \Gsubgp$ may fail to be a topological group.
  For this reason we will then define $\Gphi := \Gsubgp \rtimes
  \flow_{\R}$ to be the external semi-direct product of the universal
  cover $\flow_{\R}$ of $\bar{\flow}_\R$ which acts on the normal
  subgroup $\Gsubgp$ by inner automorphisms.
 
  We begin by explaining the algebraic outline of the construction,
  ignoring, at first, the topological difficulty mentioned above.

  Since $\map\in \text{\rm Aff}_{\Hgroup}( \mathcal N/ \Dlattice)$,
  its automorphism part $A_\map:=\text{\rm Aut}(\map)$ belongs to
  $\Hgroup$ and maps the quasi-lattice $\Dlattice$ into itself.  We
  claim that the product $\bar{\Dlattice}:= \Dlattice\cdot (
  A_\map^{im})_{i\in\Z} $ is a subgroup of $ \bar{\Gsubgp}$. In fact,
  since $A^m_\map = \text{\rm Aut}(\bar \phi_{m})$ and the translation
  part of $\bar{\phi}_{m} = \map^m$ belongs to $\Gsubgp$, it follows
  that $A^m_\map \in \bar {\mathcal N}$. Since $\Dlattice$ is a
  subgroup in $\Gsubgp$, it follows that $\bar{\Dlattice}$ is a
  subgroup of $\bar {\Gsubgp}= \phi_\R \circ \Gsubgp$.
  
  \smallskip Let $\bar F$ be defined by
  \[
  \bar F ( x) = \hatf \circ x \circ A_\map^{-1}\,, \quad \text{ for
    every }\, x \in \bar {\Gsubgp} \,.
  \]
  The map $\bar{F}$ is a diffeomorphism of $\bar{\Gsubgp}$ onto
  itself. In fact, by construction, the affine map $\hatf$ coincides
  with $\phi_{-1} \circ \map $ and commutes with the group
  $\bar{\phi}_\R$. Furthermore $ \map \circ n = \map( n ) \circ
  A_\map$, for all $n \in \Gsubgp$.  Hence for all $(t, n) \in\R
  \times \Gsubgp$, we have
  \begin{equation}
    \label{eq:action}
    \begin{aligned}
      \hatf \circ \bar \phi_t \circ n \circ A_\map^{-1} &=
      \bar{\phi}_t \circ \hatf \circ n \circ A_\map^{-1} =\bar{\phi}_t
      \circ \bar\phi_{-1} \circ \map \circ n \circ A_\map^{-1} \\ &=
      \bar{\phi}_{t -1} \circ \map ( n) \in \bar {\mathcal N} \,.
    \end{aligned}
  \end{equation}

  Since the diffeomorphism $\bar{F}$ commutes with the action of the
  subgroup $\bar {\Dlattice}$ by multiplication on the right, since
  $\hatf$ has finite order $m\in \N$ in the group $\text{\rm
    Aff}_{\Hgroup}(\Ggroup)$ and since $ A_\map^{-m} \in
  \bar{\Dlattice}$, it follows that the diffeomorphism $\bar{F}$
  induces a periodic diffeomorphism $\bar{F}_m$ of order $m$ of the
  quotient space $\bar{\Gsubgp}/ \bar{\mathcal D}$.

  \begin{sloppypar}
    Let $\bar{\finitegp}\approx \Z/m\Z$ be the group of
    diffeomorphisms of $\bar{\Gsubgp}/ \bar{\mathcal D}$ generated by
    $F_m$. The one-parameter group $\bar{\phi}_\R \subset
    \bar{\mathcal N}$ acts by left translations on the quotient
    $\bar{\finitegp} \backslash \bar{\mathcal N}/\bar{\Dlattice}$
    since the map $\bar{F}$ commutes with the left multiplication by
    the one-parameter group~$\bar \phi_\R$. By the above formula, it
    is also clear that the left translation by the element $\bar
    \phi_1$ on $\bar{\finitegp} \backslash
    \bar{\Gsubgp}/\bar{\Dlattice}$, restricted to the subset $\mathcal
    N/\Dlattice$, coincides with the map $\map$ on $\mathcal
    N/\Dlattice$.
  \end{sloppypar}

  \smallskip The map $\bar{F}$ is an affine diffeomorphisms of
  $\bar{\Gsubgp}$. In fact, the inner automorphism $ \text{\rm
    Int}(\map)$ determined by $\map$ on $\text{\rm
    Aff}_{\Hgroup}(\Ggroup)$ induces an automorphism of
  $\bar{\Gsubgp}$ as shown by the following formula. For $j\in \Z$,
  let $n_{\map^j}\in \Gsubgp$ be the translation part of the affine
  map $\map^j$ which we write as $\map^j= n_{\map^j}\circ A^j_\map$.
  For all $(t, n) \in\R \times \Gsubgp$ we have
  \[
  \begin{aligned}
    \text{\rm Int}(\map) ( \bar{\phi}_t \circ n )&= \map \circ
    \bar{\phi}_t \circ n \circ \map^{-1} \\ &= \bar{\phi}_t \circ \map
    \circ n \circ \map^{-1} = \bar{\phi}_t \circ \map( n) \circ
    n_\map^{-1} \in \bar{\Gsubgp} \,.
  \end{aligned}
  \]
  It follows that the map $\bar{F}$ is the composition of an
  automorphism of $ \bar{\Gsubgp}$ (restriction of the inner
  automorphism $\text{\rm Int}(\map)$ of $\text{\rm Aff}_{\Hgroup}
  (\Ggroup)$ given by $\map$), followed by a left translation (by the
  element $\bar{\phi}_{-1} \in \bar{\Gsubgp}$) and by a right
  translation (by the element $n_\map \in \bar{\Gsubgp}$), that is,
  for all $x\in \bar{\Gsubgp}$,
  \begin{equation}
    \label{eq:barPsi}
    \bar{F}( x)  :=  \bar{\phi}_{-1}  \circ  \text{\rm Int}(\map)
    (x)  \circ n_\map \,.
  \end{equation}
  As any right translation is an affine map given by an inner
  automorphism followed by a left translation, we have proved that the
  diffeomorphism $\bar{F}$ is affine.

  \smallskip Let us now take care of the topological part of the
  construction. As the universal cover of the group $\bar\flow_\R$ is
  isomorphic to $\R$, we define the group $\Gphi:= \Gsubgp\rtimes\R$
  to be the external semi-direct product of the group $\R$ which acts
  on the normal sub\-group~$\Gsubgp$ via the action of the
  one-parameter group $\bar{\phi}_\R$ by inner automorphisms.  In
  other terms, by definition the group $\Gphi$ is the Cartesian
  product $ \Gsubgp\times\R $ endowed with the following product law:
  for all $(n_1, t_1)$, $(n_2,t_2) \in \mathcal N \times\R$, we let
  \[
  (n_1,t_1) \ast (n_2,t_2)= ( n_1\circ {\bar \phi}_{t_1} \circ
  n_2\circ {\bar \phi}_{t_1}^{-1} ,t_1+t_2)\,.
  \]
  The group $\Gphi$ endowed with the product topology is a connected
  Lie group.  By construction there exists a group epimorphism $\pi:
  (n,t)\in \Gphi \mapsto n\circ \bar \phi_t\in \bar{\Gsubgp}$. Let
  $\biglattice := \pi^{-1} (\bar{\Dlattice})$ be the inverse image of
  the subgroup $\bar{\Dlattice} \subset \bar{\Gsubgp}$.  By definition
  of the subgroup $\bar{\Dlattice}$, the group $ \biglattice$ is
  generated by $ \Dlattice \times \{0\}$ and by the elements
  $A_\phi^{km}= n_{\psi^{km}}^{-1}\circ
  \map^{km}=n_{\psi^{km}}^{-1}\circ \bar\phi_{km} $ hence given by
  following formula:
  \begin{equation}
    \label{eq:Dphi}
    \biglattice = \{ (\gamma n_{\psi^{km}}^{-1},km) \mid k\in \Z, \gamma \in D\}.
  \end{equation}
  It follows that the subgroup ${\biglattice}$ is a quasi-lattice in
  $\Gphi$ as its projection on the second coordinate is the lattice
  $m\Z$ in $\R$, and for each $t\in \R$ the fiber over $t$ of
  $\Gphi/\biglattice$ is the finite volume space $\Gsubgp/({\bar
    \phi}_{t} \Dlattice {\bar \phi}_{t}^{-1})$.

  The one-parameter group $\flow_\R$ defined by $\flow_t=(1,t)$, for
  all $t\in\R$, acts by left multiplication on $\Gphi/\biglattice$,
  hence defining a flow on $\Gphi/\biglattice$ denoted by the same
  letter.
 
  \smallskip For all $ (1_{\Gsubgp},t), (n,0) \in \Gphi$, let
  \begin{equation}
    \label{eq:Psi}
    F\big((1_{\Gsubgp},t)*(n,0)\big) :=
    (1_{\Gsubgp},t-1)*(\map (n),0) \,.
  \end{equation}
  It is clear, by the definition, that $F\colon \Gphi\to \Gphi$ is a
  diffeomorphism commuting with the flow of left translations by
  $\phi_\R$. Furthermore, formula \eqref{eq:barPsi} shows that the map
  $F$ projects, via the epimorphism $\pi: \Gphi \to \bar{\Gsubgp}$, to
  the map $\bar F\colon \bar\Gsubgp\to \bar\Gsubgp$.

  We claim that the map $F$ induces a periodic diffeomorphism $F_m$ of
  order $m$ of the quotient space $\Gphi/ {\biglattice}$ and we denote
  by $\finitegp\approx \Z/m\Z$ the finite group of affine
  diffeomorphims of $ \Gsubgp/\Dlattice$ generated by $F_m$. To see
  this, it suffices to verify that $F^m$ induces the identity map
  on~$\Gphi/ {\biglattice}$. In fact, for all $ (1_{\Gsubgp},t), (n,0)
  \in \Gsubgp\rtimes \R$, we have
  \[
  \begin{aligned}
    F^m \big((1_{\Gsubgp},t)*(n,0)\big)*\biglattice&=
    (1_{\Gsubgp},t-m)*(\map^k (n),0)*(n^{-1}_{\psi^m},m)*\biglattice\\
    &= (1_{\Gsubgp},t)*(1_{\Gsubgp},-m)*(\map^k (n)
    n^{-1}_{\map^m},0)*(1_{\Gsubgp},m)*\biglattice\\&=
    (1_{\Gsubgp},t)*(1_{\Gsubgp},-m)*(\text{\rm Int}(\map^k) (n),0)*
    (1_{\Gsubgp},m)*\biglattice\\
    &= (1_{\Gsubgp},t)*(n,0)*\biglattice.
  \end{aligned}
  \]

  It follows from formula~\eqref{eq:Psi} that the points
  $(1_{\Gsubgp},t)*(n,0)\big)= \flow_{t}\big((n,0)\big)$ and $
  \flow_{t-1}\big((\map(n),0)\big)$ are identified by the map $F$, and
  therefore by the group $\<F\>$ of affine diffeomorphisms generated
  by $F$.  Thus, in the quotient space $\<F\> \backslash \Gphi$, the
  first return map of the flow $\phi_\R$ to the transverse section
  $\Gsubgp \times \{0\} \approx \Gsubgp$ coincides with the time-$1$
  map of the flow and is conjugated to the affine map $\map$ on
  $\Gsubgp$ in the sense that
  \[
  F\big(\phi_{1} (n,0)\big) = \left( \map (n),0\right)\,, \quad \text{
    \rm for all } n \in \Gsubgp\,.
  \]
  By passing to the quotient by $\Dlattice$, it follows that the
  return map of the flow $\phi_\R$ to the transverse section
  $(\mathcal N\times \{0\} )/\biglattice \approx \Gsubgp/\Dlattice$ in
  the double-coset space $ \finitegp\backslash \Gphi/\biglattice$
  coincides with the time-$1$ map and is conjugated to the affine map
  $\map$ on $\Gsubgp/\Dlattice$.

  \smallskip Let $\mathcal L$ be a Levi subgroup of the group $\Gphi$
  and let $q\colon \Gphi\to \mathcal L$ be the projection of group
  $\Gphi$ onto $\mathcal L$, with kernel the radical of $\Gphi$.  As
  the radical is a characteristic subgroup of $\Gphi$, any affine map
  of $\Gphi$ projects, via the homomorphism $q$, to an affine map of
  the Levi subgroup $\mathcal L$.  In particular the affine
  diffeomorphism $F$ and the one-parameter group $\flow_{\R}$ project
  to an affine diffeomorphism $q(F)$ of $\mathcal
  L/\overline{q(\biglattice)}$ and to a one-parameter group
  $q({\flow}_{\R})< \mathcal L$, inducing commuting actions on
  $\mathcal L/\overline{q(\biglattice)}$.

  \smallskip The argument is therefore completed.
\end{proof}

\begin{remark}
  \label{rem:3:1}
  To prove Theorem~\ref{thm:Main}, we may limit ourselves to
  consider the case where the flow~$\flow_\R$ is ergodic on $G/D$ with
  respect to a finite $G$-invariant measure. This is due to the fact
  that the ergodic components of the flow~$\flow_\R$ are closed
  subsets of~$G/D$ (see~\cite[Thm.\ 2.5]{Starkov}). Since $G/D$ is
  connected, if he flow~$\flow_\R$ is not ergodic, then it has
  infinitely many ergodic components, in which case
  Theorem~\ref{thm:Main} follows.
 
  The proof of Theorem~\ref{thm:Main_Affine} relies, after some
  technical reductions, on Theorem~\ref{thm:Main} and
  Lemma~\ref{lemma:reduct_transl}.  

Let us explain the main difficulty
  in applying the Lemma~\ref{lemma:reduct_transl}.  We recall that, by
  Proposition~\ref{prop:3homogflow:1}, the space of invariant
  distribution for a suspension flow $(\susp{M},\flow_\R)$ of a
  diffeomorphism $(M,\map)$ is isomorphic to the space of invariant
  distribution for the diffeomorphism $\map$.  However, the flow
  $(\Gphi/\biglattice, \flow_\R)$ constructed in
  Lemma~\ref{lem:3homogflow:1} is the suspension of \emph{a power} of
  the affine diffeomorphism~$\map$ on $\Gsubgp/\Dlattice$. The actual
  suspension of the diffeomorphism $\map$ is given by the projected
  flow $(\finitegp\backslash\Gphi/\biglattice, \flow_\R)$. Now,
  invariant distributions for the flow $(\Gphi/\biglattice,
  \flow_\R)$, produced, for example, by Theorem~\ref{thm:Main},
  may vanish when projected on the space
  $\finitegp\backslash\Gphi/\biglattice$ and special care must be
  taken to avoid this problem. We remark, however, that this
  difficulty does not arise when the invariant distributions for the
  flow $(\Gphi/\biglattice, \flow_\R)$ are given by measures,
  since the projections of ergodic probability measures of a flow
  under a finite-to-one projection are ergodic probability measures
  for the projected flow. Thus, in applying the
  Lemma~\ref{lemma:reduct_transl}, we may suppose that the homogeneous
  flow $(\Gphi/\biglattice, \flow_\R)$ there constructed is
  ergodic, by the same argument applied at the beginning of this
  remark.
\end{remark}

Henceforth we shall consider ergodic flows. Whenever convenient we may
also assume that $G$ is simply connected by pulling back the isotropy
group $D$ to the universal cover of $G$. We remark that if $D$ is a
quasi-lattice this pull-back is also a quasi-lattice.

Let $G=L\ltimes R$ be the Levi decomposition of a simply connected Lie
group $G$ and let $G_\infty$ be the smallest connected normal subgroup
of $G$ containing the Levi factor~$L$. Let $q: G\to L$ be the
projection onto the Levi factor. We shall use the following result.

\begin{theorem}[{\cite{MR0444835}, \cite{MR896893}, \cite[Lemma 9.4,
    Thm.\ 9.5]{Starkov}}]
  \label{thm:3:2}
  If $G$ is a simply connected Lie group and the flow $(G/D,
  \flow_\R)$ on the finite volume space $G/D$ is ergodic then
  \begin{itemize}
  \item The groups $R/R\cap D$ and $q(D)$ are closed in $G$ and in $L$
    respectively. Thus $G/D$ factors onto $L/q(D)\approx G/RD$ with
    fiber $R/R\cap D$.  The semi-simple flow $(L/q(D), q(\flow_\R))$
    is ergodic.
  \item The solvable flow $(G/\overline{G_\infty D}, \flow_\R)$ is
    ergodic.
  \end{itemize}
\end{theorem}

By Theorem~\ref{thm:3:2} it is possible to reduce the analysis of the
general case to that of the semi-simple and solvable cases.  In fact,
the following basic result holds.

\begin{lemma}\label{lem:3:3}
  Let $p: G\to G^{(1)}$ be an epimorphism with $D \subset
  p^{-1}(D^{(1)})$ and let $\hat p \colon G/D \to G^{(1)}/D^{(1)}$ be
  the induced quotient map. Let $\map\colon G/D\to G/D $ and
  $\map^{(1)}\colon G^{(1)}/D^{(1)}\to G^{(1)}/D^{(1)}$ be smooth maps
  intertwined by $\hat p $. Then the dimension of the space of
  $\map$-invariant distributions on $ G/D $ of Sobolev order at most
  $s$ is greater than or equal to the dimension of the space of
  $\map^{(1)}$-invariant distributions in $\mathcal
  D'(G^{(1)}/D^{(1)})$ of Sobolev order at most $s$.

  In particular, if the flow $(G/D, \flow_\R)$ of $G/D$ projects onto
  a flow $(G^{(1)}/D^{(1)}, \flow^{(1)}_\R)$ via the epimorphism~$p$,
  then the existence of countable many independent invariant
  distributions for the flow $(G^{(1)}/D^{(1)}, \flow^{(1)}_\R)$
  implies the existence of countable many independent invariant
  distributions for the flow $(G/D, \flow_\R)$. An analogous statement
  is valid for an affine map $\map$ projecting, via $p$, onto an
  affine map $\map^{(1)}$ of $G^{(1)}/D^{(1)}$.

\end{lemma}
\begin{proof}
  Let $H=L^2(G^{(1)}/D^{(1)})$ and let $\hat p: G/D \to
  G^{(1)}/D^{(1)}$ be the map induced by the epimorphism $p$. The map
  $\hat p$ is $G$-equivariant for the natural left action of $G$ on
  $G/D$ and $G^{(1)}/D^{(1)}$ and measure preserving for the
  $G$-invariant probability measures on these spaces. It follow that
  the pull-back map $\hat p^*$ is a $G$-equivariant isometry of $H$
  onto the $G$-invariant closed subspace $\hat p^*(H) \subset
  L^2(G/D)$. We therefore define the \emph{push-forward} map $\hat
  p_*\colon \hat p^*(H) \to H$ as the inverse of $\hat p^*$.

  Since the orthogonal decomposition $L^2(G/D) = \hat p^*(H) \oplus
  \hat p^*(H) ^\perp$ is $G$-invariant, for any smooth function $f$ on
  $G/D$ its components in this orthogonal decomposition are smooth. It
  follows that the push-forward map $\hat p_*$ is a linear map of the
  space of compactly supported smooth functions $C^\infty_0(G/D)$ onto
  a dense subspace of $C^\infty_0(G^{(1)}/D^{(1)})$. Setting, for any
  $D\in \Dlattice'(G^{(1)}/D^{(1)})$ and any $f\in C^\infty_0(G/D)$,
  \[
  \hat p^*(D)(f)= D(\hat p_*(f)),
  \]
  we obtain a linear continuous injection of $\mathcal
  D'(G^{(1)}/D^{(1)})$ into $\Dlattice'(G/D)$. This map preserves the
  Sobolev order of distributions because it is $G$-equivariant.

  For any pair of smooth maps $\map$ and $\map^{(1)}$ of $G/D$ and
  $G^{(1)}/D^{(1)}$ respectively, such that $\hat p\circ
  \map=\map^{(1)}$, if $D\in \Dlattice'(G^{(1)}/D^{(1)})$ is
  $\map$-invariant, the image distribution $\hat p^*(D)$ is
  $\map^{(1)}$ invariant. Thus the dimension of the space of
  $\map$-invariant distributions in $\Dlattice'(G/D)$ is greater than
  or equal to the dimension of the space of $\map^{(1)}$-invariant
  distributions in $\Dlattice'(G^{(1)}/D^{(1)})$.
\end{proof}

In dealing with solvable groups it is useful to recall the theorem by
Mostow (see~\cite[Theorem E.3]{Starkov})

\begin{theorem}[Mostow]
  \label{thm:mostow} If $G$ is a solvable Lie group, then $G/D$ is of
  finite volume if and only if $G/D$ is compact.
\end{theorem}

When $G$ is semi-simple, in proving Theorem~\ref{thm:Main}, we may
suppose that $G$ has finite center and that the isotropy group $D$ is
a lattice. This is the consequence of the following proposition.

\begin{proposition}
  \label{prop:lattice_red}
  Let $G$ be a connected semi-simple group and let $G/D$ be a finite
  volume space.  If there exists an ergodic flow on $G/D$, then the
  connected component of the identity of $D$ in $G$ is normal in
  $G$. Hence we may assume that $G$ has finite center and that $D$ is
  discrete.
\end{proposition}

\begin{proof}
  We have a decomposition $G=K\cdot S$ of $G$ as the almost-direct
  product of a compact semi-simple normal subgroup $K$ and of a
  totally non-compact normal semi-simple group $S$. Let $p: G\to
  K^{(1)}:=G/S$ be the projection of $G$ onto the semi-simple compact
  connected group $K^{(1)}$.  Let $\bar \flow_t$ the flow generated by
  $\bar X=p_* X$ on the connected, compact, Hausdorff space
  $Y:=K^{(1)}/\overline{p(D)}$. As $Y$ is a homogeneous space of a
  compact semi-simple group, the fundamental group of $Y$ is
  finite. The closure of the one-parameter group $(\exp{t\bar
    X})_{t\in \R}$ in $K^{(1)}$ is a torus subgroup $T<K^{(1)}$; it
  follows that the closures of the orbits of $\bar \flow_t$ on $Y$ are
  the compact tori $Tk \,\overline{p(D)}$, ($k\in K^{(1)}$),
  homeomorphic to $T/T\cap k \,\overline{p(D)}\, k^{-1}$.

  Let $\flow_\R$ be an ergodic flow on $G/D$, generated by $X\in
  \mathfrak g_0$.  Since $\bar \flow_t$ acts ergodically on $Y$, the
  action of $T$ on $Y$ is transitive. In this case we have $Y=T/T\cap
  \overline{p(D)}$, and since the space $Y$ is a torus with finite
  fundamental group, it is reduced to a point. It follows that $T<
  \overline{p(D)}=K^{(1)}$. Thus $p(D)$ is dense in $K^{(1)}=G/S$ and
  $SD$ is dense in $G$.

  Let $\bar D^Z$ denote the Zariski closure of $\Ad(D)$ in $\Ad(G)$
  (we refer to Remark 1.6 in \cite{MR591617}).  By Borel Density Theorem
  (see~\cite[Thm.~4.1, Cor.~4.2]{MR591617} and \cite{MR2158954}) the
  hypothesis that $G/D$ is a finite volume space implies that $\bar
  D^Z$ contains all hyperbolic elements and unipotent elements
  in~$\Ad(G)$. As these elements generate $\Ad(S)$, we have $\Ad(S) <
  \bar D^Z $, and the density of $SD$ in $G$ implies $\Ad(G)= \bar D^Z
  $. Since the group of $\Ad(g)\in \Ad(G)$ such that $\Ad(g)
  (\operatorname{Lie}(D)) = \operatorname{Lie}(D)$ is a Zariski-closed
  subgroup of $\Ad(G)$ containing $\bar D^Z$, we obtain that the
  identity component $D^0$ of~$D$ is a normal subgroup of~$G$ and
  $G/D\approx (G/D^0)/(D/D^0)$.  We have thus proved that we can
  assume that $D$ is discrete. We can also assume that $G$ has finite
  center since $D$ is a lattice in~$G$ and therefore it meets the
  center of $G$ in a finite index subgroup of the center.  This
  concludes the proof.
\end{proof}

Our proof of Theorem~\ref{thm:Main} considers separately the cases of
quasi-unipotent and the partially hyperbolic flows.  We recall the
relevant definitions.  \smallskip

Let $X$ be the generator of the one-parameter subgroup $\flow_\R$ and
let~$\mathfrak g^\mu$ denote the generalized eigenspaces of eigenvalue
$\mu$ of $\ad(X)$ on $\mathfrak g = \mathfrak g_0\otimes \C$. The Lie
algebra $\mathfrak g$ is the direct sum of the $\mathfrak g^{\mu}$ and
we have $[\mathfrak g^\mu, \mathfrak g^\nu] \subset \mathfrak
g^{\mu+\nu}$. Let
\[
\mathfrak p^0 =\sum_{\Re \mu =0} \mathfrak g^\mu, \quad \mathfrak p^+
=\sum_{\Re \mu >0} \mathfrak g^\mu,\quad \mathfrak p^- =\sum_{\Re \mu
  <0} \mathfrak g^\mu,
\]
\begin{definition}
  \label{def:qu_ph}
  A flow $\flow_\R$ on $G/D$ is called \emph{quasi-unipotent} if
  $\mathfrak g = \mathfrak p^0$ and it is \emph{partially hyperbolic}
  otherwise.  Thus the flow subgroup $\flow_\R$ is quasi-unipotent or
  partially hyperbolic according to whether the spectrum of the group
  $\Ad(\flow_t)$ acting on $\mathfrak g$ is contained in $U(1)$ or
  not.
\end{definition}

%%% Local Variables: 
%%% mode: latex
%%% TeX-master: "InvDist_affine"
%%% End: 

\section{The quasi-unipotent case}
\label{sec:4}

\subsection{The semi-simple quasi-unipotent case}
\label{sec:G-semisimple}

In this subsection we assume that the group $G$ is semi-simple and the
one-parameter subgroup $\flow_\R$ is quasi-unipotent.

\begin{definition}
  An $\sl_2(\R)$ triple $(a, n^+, n^-)$ in a Lie algebra $\mathfrak
  g_0$ is a non-zero triple satisfying the commutation relations
  \[ [a, n^\pm]= \pm n^\pm, \quad [n^+,n^-]=2a.
  \]
\end{definition}

We prove a generalized version of the Jacobson--Morozov
Lemma~\cite{MR559927}.

\begin{lemma}[Jacobson--Morozov Lemma]
  \label{lem:Jacobson-Morozov}
  Let $n^+$ be a nilpotent element in a semi-simple Lie algebra
  $\mathfrak g_0$. Assume that $n^+$ is invariant under the action of
  a compact subgroup $\finitegp$ of automorphisms of the Lie algebra
  $\mathfrak g_0$. Then we can find a semi-simple element $a\in
  \mathfrak g_0$ and a nilpotent element $n^-\in \mathfrak g_0$ such
  that $(a, n^+, n^-)$ is an $\sl_2(\R)$ triple invariant under the
  action of $\finitegp$.
\end{lemma}
\begin{proof}
  By the Jacobson-Morozov Lemma there exists a semi-simple element
  $a_0$ and a nilpotent element $n^-_0$ such that $(a_0, n^+, n^-_0)$
  is an $\sl_2(\R)$ triple.
    
  Let $m$ be the probability Haar measure on $\finitegp$. We define
  \[
  a := \int_{\finitegp} f (a_0) \,\D m(f).
  \]
  and
  \[
  n_1^- := \int_{\finitegp} f (n^-_0) \,\D m(f).
  \]
  
  By construction the element $a$ is $\finitegp$-invariant. Since
  $n^+$ is also $\finitegp$-invariant and $(a_0, n^+,n^-_0)$ is an
  $\sl_2(\R)$ triple it follows that
     \begin{align*}
      {[}a, n^+{]} &= \int_{\finitegp} {[} f(a_0), n^+{]} \,\D m(f)=
      \int_{\finitegp}   f  ({[}a_0,n^+{]}) \,\D m(f) \\
&=  \int_{\finitegp}   f  (n^+) \,\D m(f) =n^+ \,; \\
      {[}n^+, n_1^-{]} & = \int_{\finitegp} {[} n^+, f (n^-_0){]} \,\D
      m(f) = \int_{\finitegp} f ({[} n^+, n^-_0{]}) \,\D m(f) \\&=
      2\int_{\finitegp} f (a_0) \,\D m(f) = 2 a\,.
    \end{align*}  
  By Morozov's Lemma, there is an element $n_2^-$ such that
  $(a,n^+,n_2^-)$ is a $\sl_2(\R)$ triple.  Let us then define
  \[
  n^- := \int_{\finitegp} f (n^-_2) \,\D m(f).
  \]  
  Then $a$, $n^+$ and $n^-$ are $\finitegp$-invariant. Moreover,
  \begin{align*}
       {[}n^+, n^-{]} & = \int_{\finitegp} {[} n^+, f (n^-_2){]} \,\D
      m(f) = \int_{\finitegp} f ({[} n^+, n^-_2{]}) \,\D m(f)
      \\& = 2\int_{\finitegp} f (a) \,\D m(f) = 2 a \,; \\
      {[}a, n^-{]} & = \int_{\finitegp} {[} a, f (n^-_2){]} \,\D m(f)
      = \int_{\finitegp} f ({[} a, n^-_2{]}) \,\D m(f) \\&=
      -\int_{\finitegp} f (n_2^-) \,\D m(f) = -n^-\,.
   \end{align*}
  
  In conclusion, the elements $(a,n^+,n^-)$ form an $\sl_2(\R)$ triple
  (in particular $n^-$ is a non trivial nilpotent element), invariant
  under the action of the compact subgroup $\finitegp\subset \text{\rm
    Aut}(\mathfrak g_0)$.
\end{proof}

Given a unitary representation of $(\pi, H)$ of a Lie group on a
Hilbert space $H$, we denote by $H^\infty$ the subspace of
$C^\infty$-vectors of $H$ endowed with the $C^\infty$ topology, and by
$(H^{\infty})'$ its topological dual.

\begin{lemma}[\cite{FF1}]
  \label{lem:4:1}
  Let $U_t$ be a unipotent subgroup of\/ $\hbox{\rm PSL}_2(\R)$. For
  each non-trivial irreducible unitary representation $(\pi, H) $ of\/
  $\hbox{\rm PSL}_2(\R)$ there exists a distribution, i.e.\ an element
  of $D\in (H^{\infty})'$, such that $U_t D = D$.
\end{lemma}

\begin{proposition}
  \label{prop:4:2}
  Let $G$ be a semi-simple group and let $G/D$ be a finite volume
  space. Suppose that~$\flow_\R$ is a quasi-unipotent subgroup of $G$
  such that the flow of $\flow_\R$ on $G/D$ is ergodic and commutes
  with the action of a finite group $\finitegp$ of affine
  diffeomorphisms of $G/D$.  Then, there exists infinitely many
  independent $\flow_\R$-invariant distributions on $C^\infty(G/D)$ of
  Sobolev order $1/2$ which are also $\finitegp$-invariant.
\end{proposition}

\begin{proof}
  By Proposition~\ref{prop:lattice_red} we may assume that $G$ has
  finite center and that $D$ is a lattice in $G$.  By the Jordan
  decomposition we can write $\flow_t= c_t \times u_t$ where $c_t$ is
  semi-simple and $u_t$ is unipotent with $c_\R$, $u_\R$ commuting
  one-parameter subgroups of~$G$. Since $c_\R$ is semi-simple and
  quasi-unipotent, its closure in $G$ is a torus $T$.
  
  By Lemma~\ref{lem:homogflow3:1}, the group of affine diffeomorphims
  $\finitegp$ acts by conjugation by automorphism of the Lie algebra
  $\mathfrak g$ of $G$ fixing the subgroup $\flow_\R$.  By the
  uniqueness of the Jordan decomposition of $\flow_\R$, the the action
  by conjugation of $\finitegp$ fixes the subgroups $c_\R$ and $u_t$,
  and consequently it fixes the torus $T$, closure of $c_\R$ in $G$.
  Let $\hat \finitegp$ be the compact Abelian group of affine
  diffeomorphisms of $G/D$ generated by $\finitegp$ and left
  translations by elements of torus $T$.

  By the generalized version of the Jacobson--Morozov
  lemma~\ref{lem:Jacobson-Morozov} we can find $\sl_2(\R)$ triple $(a,
  n^+, n^-)$ invariant under $\hat \finitegp$. Thus the action by left
  translations on $G/D$ of the analytic group $S$ generated by the
  triple $(a, n^+, n^-)$ commutes with the action of $\hat\finitegp$.

  It is well known that the center $Z(S)$ of $S$ is finite and that,
  consequently, there exists a maximal compact subgroup~$K\approx S^1$
  of~$S$ containing $Z(S)$. (Indeed, the adjoint representation
  $\Ad_G|S$ of $S$ on the Lie algebra of $G$, as a finite dimensional
  representation of $S$, factors through $SL(2, \R)$, a double cover
  of $S/Z(S)$. The kernel of $\Ad_G|S$ is contained in $Z(G)$, because
  $G$ is connected. Since $Z(S)$ is monogenic, we have that $Z(G)$ is
  a subgroup of index one or two of $Z(G)Z(S)$. It follows that $Z(S)$
  is finite).

  The group $T^{(0)}=\hat \finitegp \cdot K$ is a compact, Abelian
  group of affine transformation of $G/D$ whose connected component of
  identity is a torus. It follows that the double coset space
  $T^{(1)}\backslash G/D$ is a non trivial orbifold and that the space
  $H^{(0)}$ of $L^2$ functions on $G/D$ which are invariant under the
  action of $T^{(0)}$ has infinite dimension. The space $H^{(0)}$ is
  contained in the space $H_{\finitegp}$ of $L^2$ functions on $G/D$
  invariant under $Z(S)$ and $\hat \finitegp$, space on which the
  group $S$ acts unitarily.
  
  As the center $Z(S)$ of $S$ acts trivially on $H_\finitegp$, the
  Hilbert space $H_\finitegp$ decompose as a direct integral $\int
  H_\alpha \D\nu_\alpha$ of irreducible unitary representations
  $H_\alpha$ of ${\mathrm P\mathrm S\mathrm L}_2(\R)$, where
  $\nu_\alpha$ is a measure on the unitary dual of ${\mathrm P\mathrm
    S\mathrm L}_2(\R)$. Since every irreducible unitary representation
  of ${\mathrm P\mathrm S\mathrm L}_2(\R)$ contains at most one
  $K$-invariant vector and since the space $H^{(0)}\subset
  H_\finitegp$ of $K$ invariant vectors is infinite dimensional, we
  deduce that the measure $\nu_\alpha$ has an infinite support, i.e.\
  that the space $H_\finitegp$ is not a finite sum of irreducible
  unitary representations of ${\mathrm P\mathrm S\mathrm L}_2(\R)$.

  By the previous Lemma each unitary irreducible representation
  $H_\alpha$ of ${\mathrm P\mathrm S\mathrm L}_2(\R)$ occurring in the
  support of $\nu_\alpha$ contains a distribution $D_\alpha\in
  (H^\infty)'$ of Sobolev order~$1/2$ which is $u_t$-invariant. Since
  $H_{\finitegp}$ consists of functions that are $\hat
  \finitegp$-invariant, this distribution is also invariant by
  translations $\flow_t=c_tu_t$ and by the affine maps in~$\finitegp$.
  Since the space $H^\infty_\finitegp$ coincides with the Fr\'echet
  space of $C^\infty$ functions on $G/D$ which are $T \cdot Z(S)$
  invariant as well as $\finitegp$-invariant, the proposition is
  proved.
\end{proof}

\subsection{The solvable quasi-unipotent case}
\label{sec:solv-eucl-case}

In this subsection we assume that the group $G=R$ is solvable and the
one-parameter subgroup $\flow_\R$ is quasi-unipotent and ergodic on
the finite volume space $R/D$.

\smallskip We recall the following definition.
\begin{definition} A solvable group $R$ is called a \emph{class
    $(\mathrm I)$ group} if, for every $g\in R$, the spectrum of
  $\Ad(g)$ is contained in the unit circle~$U(1)=\{z \in \C \mid
  |z|=1\}$.
\end{definition}

It will also be useful remark that if $R$ is solvable and $R/D$ is a
finite measure space, then we may assume that $R$ is simply connected
and that $D$ is a quasi-lattice (in the language of Auslander and
Mostow, the space $R/D$ is then a \emph{presentation}); in fact, if
$\tilde R$ is the universal covering group of $R$ and $\tilde D$ is
the pull-back of $D$ to $\tilde R$, then the connected component of
the identity $\tilde D_0$ of $\tilde D$ is simply connected
\cite[Them. 3.4]{gorbatsevich1994lie}; hence $R/D\approx \tilde R
/\tilde D \approx R'/ D' $, with $R'=\tilde R/\tilde D_0$ solvable,
connected and simply connected and with $D'=\tilde D/\tilde D_0$ a
quasi-lattice.

We also recall the construction, originally due to Malcev and
generalized by Auslander \emph{et al.}, of the \emph{semi-simple} or
\emph{Malcev splitting} of a simply connected, connected solvable
group $R$ (see \cite{MR0199308}, \cite{MR0486307}, \cite{MR0486308},
\cite{0453.22006}).

A solvable Lie group $G$ is \emph{split} if $G=N_G\rtimes T$ where
$N_G$ is the nilradical of $G$ and $T$ is an Abelian group acting on
$G$ faithfully by semi-simple automorphisms.

A \emph{semi-simple} or \emph{Malcev splitting} of a connected simply
connected solvable Lie group $R$ is a split exact sequence
\[
0 \to R \overset{m}{\to}M(R) \leftrightarrows T \to 0
\]
embedding $R$ into a split connected solvable Lie group $M(R)=
N_{M(R)} \rtimes T$ such that $M(R)= N_{M(R)} \cdot m(R)$; here
$N_{M(R)}$ and $T$ are as before. The image $m(R)$ of~$R$ is normal
and closed in $M(R)$ and it will be identified with $R$.

The semi-simple splitting of a connected simply connected solvable Lie
group $R$ is unique up to an automorphism fixing $R$.

Let $\operatorname{Aut}(\mathfrak r)\approx \operatorname{Aut}(R)$ be
automorphism group of the Lie algebra $\mathfrak r$ of $R$. The
adjoint representation $\Ad$ maps the group $R$ to the solvable
subgroup $\Ad(R)<\operatorname{Aut}(\mathfrak r)$; since
$\operatorname{Aut}(\mathfrak r)$ is an algebraic group we may
consider the Zariski closure $\Ad(R)^*$ of~$\Ad(R)$. The group
$\Ad(R)^*$ is algebraic and solvable, since it's the algebraic closure
of the solvable group $\Ad(R)$.  It follows that $\Ad(R)^*$ has a
Levi-Chevalley decomposition $\Ad(R)^* = U^*\rtimes T^*$, with $T^*$
an Abelian group of semi-simple automorphisms of $\mathfrak r$ and
$U^*$ the maximal subgroup of unipotent elements of $\Ad(R)^*$.

Let $T$ be the image of $\Ad(R)$ into $T^*$ by the natural projection
\hbox{$\Ad(R)^*\to T^*$}.  Since $T$ is a group of automorphisms of
$R$, we may form the semi-direct product $M(R)= R\rtimes T$.  By
definition we have a split sequence
\[
0 \to R \to M(R)\leftrightarrows T \to 0.
\]
It can be proved that $M(R)$ is a split connected solvable group
$N_{M(R)} \rtimes T$ and it is the semi-simple splitting of $R$ (see~\
loc.\ cit.). We remark that the splitting $M(R)=N_{M(R)} \rtimes T$
yields two projections maps $\tau\colon M(R)\to T$ and $\pi\colon
M(R)\to N_{M(R)}$ defined, for any $g\in M(R)$, by
\begin{equation}
  \label{eq:quasi-unip4:1}
  g = \tau (g) \pi(g), \qquad \tau (g)\in T,\quad \pi(g)\in N_{M(R)} \,.
\end{equation}
The projection $\tau $ an epimorphim. Composing $\tau$ with the
inclusion $R\to M(R)$ we obtain a surjective homomorphism $p:R\to T$.

The proof of the following easy lemma is omitted.

\begin{lemma}
  \label{lem:quasi-unip4:1}
  Let $G$ be a Lie group, $\Hgroup$ a subgroup of automorphisms of $G$
  and let $\hat G$ be the semi-direct product $G\rtimes \mathcal
  H$. Then any group $\finitegp$ of automorphisms $A\in \Aut(G)$
  normalizing~$\Hgroup$ extends to a group of automorphisms $\hat A\in
  \Aut(\hat G)$ by setting
  \[
  \hat A (g\cdot H) = A(g)\cdot( A H A^{-1}), \qquad \text { for all
  }A\in \finitegp
  \]
  Hence every group $\finitegp$ of affine maps $\map=uA$ of $G$, such
  that $A$ normalizes $\Hgroup$ for all $\map \in \finitegp$, extends
  to a group of affine maps $\hat \map$ of $\hat G$, defined by $\hat
  \map = u \cdot \hat A$.
\end{lemma}

\begin{proposition}
  \label{prop:quasi-unip4:2}
  Let $\finitegp$ be a finite group of automorphisms of a connected,
  simply connected, solvable Lie group $R$. For any semi-simple
  splitting $M(R)$ of $R$ the group~$\finitegp$ extends to a group of
  automorphisms of $M(R)$, that is there is a homomorphism $A\in
  \finitegp\mapsto \hat A\in\Aut(M(R))$ such that $\hat A= A$ on $R$.
\end{proposition}
\begin{proof}
  By the previous lemma, it suffice to show that, if $M(R)=R\rtimes
  T$, the group~$T$ operates on $R$ by a group of automorphisms
  normalized by $\finitegp$.

  Since, for any $A\in \Aut(R)$, we have $A\Ad(r)A^{-1} =\Ad (A(r))$,
  the group $\Ad(R)$ and its Zariski closure $\Ad(R)^*$ are normal in
  $\Aut(R)$. Let us show that if $\Ad(R)^* = U^*\rtimes T^*$ is a
  Levi-Chevalley decomposition of $\Ad(R)^*$, then the torus $T^*$ is
  normalized by any finite group $\finitegp$ of automorphisms of $R$.

  In fact since any two tori in $\Ad(R)^*$ are conjugate in
  $\Ad(R)^*$, for any $A\in \finitegp$ there exists an element $u_A\in
  U^*$ such that $A T^* A^{-1} = u_A T^* u_A^{-1}$. Thus, if we denote
  by $\operatorname{Norm}_{\,U^*}(T^*)$ the normalizer of $T^*$ in
  $U^*$ we obtain a homomorphism $\finitegp\to
  U^*/\operatorname{Norm}_{\,U^*}(T^*)$. Since the group $\finitegp$
  is finite and $U^*/\operatorname{Norm}_{\,U^*}(T^*)$ a real
  unipotent algebraic group, this homomorphism is trivial. Thus, for
  any $A\in \finitegp$, we have $A T^* A^{-1} = T^*$.

  By definition, for any $r\in R$ such that $\Ad(r) = u\cdot t$, with
  $u\in U^*$ and $t\in T^*$, the projection of $\Ad(r)$ in $T^*$ is
  equal to the element $t$. For all $A\in \finitegp$, we have $A
  \Ad(r) A^{-1} = A u A^{-1}\cdot A t A^{-1}$. Since $A t A^{-1}\in
  T^*$, we derive that the projection of $A \Ad(r) A^{-1}$ in $T^*$ is
  equal to $A t A^{-1}$. We have therefore proved that $A t A^{-1}\in
  T$, for all $t\in T$ and all $A\in \finitegp$, concluding the proof.
\end{proof}

It is useful to recall a part of Mostow's structure theorem for
solvmanifolds, as reformulated by Auslander \cite[IV.3]{MR0486307}
and~\cite[p.~271]{MR0486308}:
 
\begin{theorem}[Mostow, Auslander]
  \label{thm:mostow-auslander}
  Let $D$ be a quasi-lattice in a simply connected, connected,
  solvable Lie group $R$, and let $M(R)=R\rtimes T$ be a semi-simple
  splitting of $R$. Then $T$ is a closed subgroup of
  $\operatorname{Aut}(\mathfrak r)$ and the projection $\tau(D)$ of
  $D$ in $T$ is a lattice of $T$. If $R$ is class (I), the semi-simple
  splitting $M(R)$ of $R$ admits a structure of real algebraic group,
  with $N_{M(R)}$ its unipotent radical and $T$ a maximal torus acting
  on $N_{M(R)}$ by semi-simple automorphisms.
\end{theorem}

The following theorem was first proved in \cite[Thm.\ 4.4]{MR0199308}
under the hypothesis that $D\cap N_{M(R)}= D$. This amounts to suppose
that $D$ is nilpotent, which is the case when $R/D$ supports a minimal
flow, as it is proved in \cite[Thm.~C]{MR0486308}.  A simplification
of the latter proof under the hypothesis that $R/D$ carries an
er\-go\-dic flow appears in \cite[Theorem 7.1]{Starkov}.

\begin{theorem}[Auslander, Starkov]
  \label{thm:auslander-starkov}
  Let $\flow_\R$ be an ergodic flow on a class (I) compact solvable
  manifold $R/D$.  There exists a semi-simple splitting $M(R)=R\rtimes
  T= N_{M(R)}\rtimes T$ of $R$ such that $D < N_{M(R)}$ and the
  projection map $\pi \colon M(R)\to N_{M(R)}$ induces a
  diffeomorphism of $R/D$ onto the compact nilmanifold $ N_{M(R)}/D$
  conjugating the flow $\flow_\R$ to a nilflow.
\end{theorem}

In the following lemma we show that, for our purposes, we may assume
that $R$ is a class (I) solvable group and give a new proof that $D$
is a subgroup of the unipotent radical of the algebraic splitting.

\begin{lemma}
  \label{lem:4:3}
  If the flow $\flow_\R$ is ergodic and quasi-unipotent on the finite
  volume solvmanifold $R/D$, then the group $R$ is of class (I).
\end{lemma}

\begin{proof}
  We may assume $R$ simply connected and connected and $D$ a
  quasi-lattice. Let $M(R)=N_{M(R)} \rtimes T = R\rtimes T$ be the
  semi-simple splitting of $R$.  Since the one-parameter subgroup
  $\flow_\R$ is quasi-unipotent, the closure of the projection
  $\tau(\flow_\R)$ in the semi-simple factor $T
  <\operatorname{Aut}(\mathfrak r)$ is a compact torus $T' < T$. The
  surjective homomorphism $\tau:R\to T$ induces a continuous
  surjection $R/D \to T/\tau(D)$. By Mostow's structure Theorem $R/D $
  is compact hence $T/\tau(D)$ is a compact torus.  The orbits of $T'$
  in $T/\tau(D)$ are finitely covered by $T'$. But $\flow_\R$ acts
  ergodically on $R/D$, hence $T'$ acts ergodically and minimally on
  $T/\tau(D)$. It follows that $T/\tau(D)$ consists of a single $T'$
  orbit and, since $T', T$ are both connected and $T'<T$ and $T' \to
  T'/\tau(D)$ is a finite cover, we obtain that $T'=T$. Thus $T$
  consist of quasi-unipotent automorphisms of $\mathfrak r$, which
  implies that for all $g\in R$, the automorphism $\Ad(g)$ is
  quasi-unipotent. Hence the group $R$ is of class (I).

  The subgroup $\tau(D)$ is a finite group commuting with an ergodic
  one-parameter flow on the torus $T$. It follows that $\tau(D)$ is
  trivial, i.e.\ that $D$ is included in the unipotent radical
  $N_{M(R)}$ of the (algebraic) semi-simple splitting $M(R)$ of $R$.
\end{proof}

\begin{proposition}
  \label{prop:quasi-unip4:1}
  Let $R$ be a solvable, simply connected, connected Lie group and~$D$
  a quasi-lattice in $R$. Let $\flow_\R< R$ be a one-parameter group
  and let $\finitegp$ be a finite group of affine maps of $R/D$
  commuting with $\flow_\R$.  Assume that the flow of $\flow_\R$ is
  ergodic and quasi-unipotent.  Then there exists a semi-simple
  splitting $M(R)=R\rtimes T= N_{M(R)}\rtimes T$ of $R$ with the
  following properties
  \begin{itemize}
  \item We have $D < N_{M(R)}$. Hence the projection map $\pi \colon
    M(R)\to N_{M(R)}$, defined by \eqref{eq:quasi-unip4:1}, is a the
    identical group isomorphism of $D$.
  \item The map $\pi$ induces a diffeomorphism of $R/D$ onto the
    compact nilmanifold $ N_{M(R)}/D$ conjugating the flow $\flow_\R$
    to a nilflow $u_\R$ on $ N_{M(R)}/D$
  \item The group $\finitegp$ projects via $\pi$ to a finite group
    $\overline{\finitegp}$ of affine maps on $ N_{M(R)}/D$ commuting
    with $u_\R$.
  \end{itemize}
\end{proposition}
\begin{proof}
  The flow $\flow_\R$ is ergodic and quasi-unipotent on the finite
  volume solvmanifold $R/D$ hence, by Lemma~\ref{lem:4:3}, the group
  $R$ is of class (I). The theorem of Mostow and
  Auslander~\ref{thm:mostow-auslander} states that $R$ embeds into the
  algebraic solvable group $M(R)$. Thus, we may consider, for all
  $t\in \R$, the Jordan decomposition $\flow_t= a_t u_t$ of the
  element $\flow_t\in M(R)$: here $a_\R$ and $u_\R$ are commuting
  one-parameter groups, respectively, of semi-simple and unipotent
  elements.  Since every semisimple element is included in a torus of
  $M(R)$, the Levi-Chevalley decomposition $M(R) = N_{M(R)} \rtimes
  T$, may be chosen so that $a_\R< T$. Since $R$ is a class (I)
  solvable group, the torus $T$ is the closure of the one-parameter
  group $a_\R$.

  For any $\map= gA\in \finitegp$ set $A_{\map}:=A$.  By the
  Proposition~\ref{prop:quasi-unip4:2}, the finite group of
  automorphisms of $R$ defined by $\{ A_{\map}\in\Aut(R)\mid \map\in
  \finitegp\}$ extends to a finite group $\{ \hat
  A_{\map}\in\Aut(M(R))\mid \map\in \finitegp\}$.

  Let $\map= gA \in \finitegp$. As the map $\map$ commutes with the
  one-parameter group $\flow_\R$ for all $t\in \R$ we have the
  identity $g A (\flow_t) g^{-1} = \flow_t=g \hat A (\flow_t) g^{-1}$,
  which implies that the affine map $\hat \map:=g\hat A$ of $M(R)$
  commutes with the flow of left translation by the one-parameter
  group $\flow_\R$ on $M(R)$. This identity also implies, by the
  uniqueness to the Jordan decomposition, the identities
  \[
  g \hat A (a_t) g^{-1} = a_t\qquad\qquad g\hat A (u_t) g^{-1} = u_t\
  \]
  and therefore
  \begin{equation}
    \label{eq:quasi-unip4:2}
    g \hat A (z) = z g,\quad \text{for all $z\in T$}.
  \end{equation}
  Since the affine maps $\map=gA\in\finitegp$ induce affine maps of
  $R/D$ we have $A(D)=D$, hence $\hat A (D)=D$.  Thus the affine map
  $\hat \map$, passes to the quotient $M(R)/D $. Remark also that,
  since $N_{M(R)}$ is the unipotent radical of $M(R)$, any
  automorphism of $M(R)$ maps $N_{M(R)}$ to itself.

  It is now easy to prove Auslander and Green theorem: the map
  $p\colon R\to N_{M(R)}$ induces a diffeomorphism $\bar p\colon R/D
  \to N_{M(R)}/D$ intertwining the flow of left translation by
  $\flow_\R$ with the flow of left translations by $u_\R$.  Let us
  show that the diffeomorphism~$\bar p$ also intertwines the group of
  affine maps $\finitegp$ with a group of affine maps
  $\overline{\finitegp}:=\{\overline\map=\pi( g) \hat A\mid gA\in
  \finitegp\}$ of $N_{M(R)}/D$. Let $\map=gA\in\finitegp$ and set
  $\overline\map=\pi( g) \hat A$; for any $h\in R$ we have
  \[
  \begin{split}
    \bar p(\map(h D)) &= \pi( g A(h) D) = \pi\Big(g A\big( \tau(h)
    \pi(h)\big) D\Big)\\&= \pi\Big(\hat A\big( \tau(h)\big) g\hat
    A\big(\pi(h)\big)D\Big)= \pi( g) \hat A\big(\pi(h)D\big)
    =\overline\map (\bar p(hD))
  \end{split}
  \]
  Since the action by left translation by the one-parameter group
  $\flow_\R< R$ on $R/D$ commutes with the action of the group
  $\finitegp$, and since these actions are mapped by $\bar p$,
  respectively, to the action by left translation by the one-parameter
  group $u_\R< R$ and to the action of the group
  $\overline{\finitegp}$ of affine maps of $N_{M(R)}/D$, the proof is
  completed.
\end{proof}

\begin{lemma}
  \label{lemma:finite_order_transl}
  Any finite order affine diffeomorphism $\map$ of a nilmanifold which
  commutes with an ergodic flow is a translation by an element of the
  center.
\end{lemma}
\begin{proof} Let $\map=gA$ be an affine map of a nilmanifold $N/D$
  which commutes with an ergodic flow $\flow_\R$ on $N/D$.  Let $\bar
  N := N/[N,N]$ and $\bar D= [D,D]$. Since $[D,D]$ and $[N,N]$ are
  characteristic groups of $D$ and $N$, respectively, the affine map
  $\map$ yields, by projection, an affine map $\bar \map$ of $\bar N/
  \bar D$. The map $\bar \map$ on the torus $\bar N/ \bar D$ commutes
  with the ergodic flow $\bar \flow_\R$ on $\bar N/ \bar D$,
  projection of the flow $\flow_\R$ on $N/D$. If an affine map of a
  torus commutes with an ergodic flow, then its automorphism part is
  the identity, since a toral automorphism fixing a vector with
  rationally independent coordinates is the identity.

  Let $A_\ast \in \Aut(\mathfrak n)$ denote the automorphism of the
  Lie algebra $\mathfrak n$ of $N$ induced by $A\in \Aut(N)$.

  Let $\{ \mathfrak n^{(k)}\}$ denote the descending central series of
  $\mathfrak n$ defined by induction as $\mathfrak n^{(0)}= \mathfrak
  n$ and $\mathfrak n^{(k+1)}= [\mathfrak n^{(k)}, \mathfrak n]$ for
  all $k\in \N$.

  Since $A_\ast$ projects to the identity map on the Abelianized Lie
  algebra $\mathfrak n/[\mathfrak n,\mathfrak n]$, we can write
  $A_\ast=I+ L_1$ for some linear map $L_1\colon \mathfrak n \to
  \mathfrak n^{(1)}$. Assume, by recurrence on $i$, that we have
  $A_\ast=I+ L_i$ where $L_i$ is a linear map from $\mathfrak n$ to
  $\mathfrak n^{(i)}$, Then, for $x=[y,z]$ with $y\in \mathfrak
  n^{(i-1)}$ and $z\in \mathfrak n $, we have
  \[
  A_\ast x = [A_\ast y,A_\ast z]= [y+L_iy, z+L_i z]= x +x'
  \]
  with $x'\in \mathfrak n^{(i+1)}$.  It follows that $A_\ast=I+
  L_{i+1}$ where $L_{i+1}$ is a linear map from $\mathfrak n$ to
  $\mathfrak n^{(i+1)}$. For $i$ equal to the degree of nilpotency of
  $N$, we conclude that $A_\ast$ is the identity automorphism of
  $\mathfrak n$ and that the affine map $\map$ is a translation.

  \smallskip We claim that any finite order translation of a
  nilmanifold is a translation by an element of the center.  In fact,
  a translation of a nilmanifold $N/D$ by an element $m\in N$ is equal
  to the identity if and only if
  \[
  n^{-1} m n \in D, \quad \text{ for all }\, n \in N.
  \]
  It follows that $m\in D\cap Z(N)$, where $Z(N)$ is the center of
  $N$.

  Thus if the $k$-th power of a translation by $m\in N$ is the
  identity of $N/D$ we have that $m^k\in D\cap Z(N)$. Since the
  exponential map is onto, we conclude that $m$ belongs to the center
  $Z(N)$ as claimed.
\end{proof}

\begin{lemma} 
  Any translation of a non-toral nilmanifold has infinitely many
  invariant distributions of Sobolev order $1/2$.
\end{lemma}

\begin{proof} Since the exponential map of any nilpotent group is
  surjective, every translation of a nilmanifold is the return map
  (with constant return time) of a nilflow. Since the suspension of a
  non-toral nilmanifold by a translation is non-toral, the suspension
  flow, hence its return map, has infinitely many independent
  invariant distributions of Sobolev order $1/2$.
\end{proof}
 
\begin{lemma}
  \label{lem:4quasi-unip:1}
  Let $M$ be a closed connected submanifold of a torus $\T^d$
  transverse to a linear minimal flow $(\flow_t)$ such that the return
  time of the flow to $M$ is everywhere constant (assume is equal to
  $1$). Then $M$ is a subtorus and the map $\flow_1$ is a constant
  translation on this subtorus.
\end{lemma}
\begin{proof}
  Let $x\in M$ and let $x_n = \flow_n (x)$. Since the translation
  $\flow_n$ maps $M$ into itself, we have that $T_x M = T_{x_n} M
  \subset \R^d$.  Since the set ${x_n}$ is dense in $M$, by continuity
  ($M$ needs to be at least $C^1$) we have that $T_y M =E\subset \R^d$
  is a constant space $E$ for all $y\in M$.  It follows that $M$
  coincides locally with a translate of the projection of $E$ to the
  torus, hence there exits a translate $T'$ of the projection of $E$
  to the torus such that the set $T'\cap M$ non-empty, open and
  closed. Since $M$ is connected it follows that $M=T'$ and since $M$
  is closed, it is a subtorus.
\end{proof}

The first two authors have proved that the main theorem holds for
general nilflows, that is, that the following result holds:
\begin{theorem}[\cite{FF}]
  \label{thm:flaminio-forni}
  An ergodic nilflow which is not toral, has countably many
  independent invariant distributions of Sobolev order $1/2$.
\end{theorem}

In conclusion we have

\begin{proposition}
  \label{prop:4:4}
  An ergodic quasi-unipotent flow on a finite volume solvmanifold is
  either smoothly conjugate to a linear toral flow or it admits
  countably many independent invariant distributions of Sobolev order
  $1/2$.
\end{proposition}

%%% Local Variables: 
%%% mode: latex
%%% TeX-master: "InvDist_affine"
%%% End: 

\section{Partially hyperbolic homogeneous flows}
\label{sec:5}

In the non-compact, finite volume case, by applying results of
D.~Kleinbock and G.~Margulis we are able to generalize
Theorem~\ref{expandingminimal} to flows on \emph{semi-simple}
manifolds.  We think that it is very likely that a general partially
hyperbolic flow on any finite volume manifold has infinitely many
different minimal sets, but we were not able to prove such a general
statement.

For non-compact finite volume we recall the following result by
D.~Kleinbock and G.~Margulis \cite{MR96k:22022}

and its immediate corollary.

\begin{theorem}[Kleinbock and Margulis] \label{thm:kleinbock-margulis}
  Let $G$ be a connected semi-simple Lie group of dimension $n$
  without compact factors, $\Gamma$ an irreducible lattice in $G$. For
  any partially hyperbolic homogeneous flow $\flow_\R$ on $G/\Gamma$,
  for any closed invariant set $Z\subset G/\Gamma$ of (Haar) measure
  zero and for any nonempty open subset $W$ of $G/\Gamma$, we have
  that
  \[
  \dim_H(\{x\in W\;|\;\flow_\R x\;\mbox{is bounded
    and}\;\overline{\flow_\R x}\cap Z=\emptyset\})=n.
  \]
  Here $\dim_H$ denotes the Hausdorff dimension.
\end{theorem}

Observe that if the flow $\flow_\R$ is ergodic, it is enough to assume
that the closed invariant set $Z \subset G/\Gamma$ be proper.

\begin{corollary}\label{cor:kleinbock-margulis}
  Under the conditions of Theorem~\ref{thm:kleinbock-margulis} the
  flow $(G/\Gamma,\flow_\R)$ has infinitely many different compact
  minimal invariant sets.
\end{corollary}
 
\begin{proposition}\label{prop:5:2}
  Let $G$ be a connected semi-simple Lie group and $G/D$ a finite
  volume space. Assume that the flow $\flow_\R$ on $G/D$ is partially
  hyperbolic.  Then the flow $(G/\Gamma,\flow_\R)$ has infinitely many
  distinct minimal invariant sets supporting infinitely many
  $g_\R$-invariant and mutually singular ergodic probability measures.
\end{proposition}
 
\begin{proof}
  Let $G=K\times S$ be the decomposition of $G$ as the almost-direct
  product of a compact semi-simple subgroup $K$ and of a totally
  non-compact semi-simple group $S$, with both $K$ and $S$ connected
  normal subgroups. Since the flow $(G/\Gamma,\flow_\R)$ is partially
  hyperbolic $S$ is not trivial. Since $K$ is compact and normal, then
  $D'=DK=KD\subset G$ is a closed subgroup, and since $D\subset KD$,
  then $G/KD$ is of finite volume. Moreover, $(G/K)/DK\sim G/KD$ is of
  finite volume and $G'=G/K\sim S/S\cap K$ is semi-simple without
  compact factor with $D'\subset G'$ a closed subgroup with $G'/D'$ of
  finite volume and a projection $p:G/D\to G'/D'$.  Thus we may assume
  that $G$ is totally non compact, and by
  Proposition~\ref{prop:lattice_red}, that the center of $G$ is finite
  and that $D$ is a lattice.

  If $D$ is irreducible, then the statement follows immediately from
  Corollary~\ref{cor:kleinbock-margulis}.  Otherwise, let $G_i$, for
  $i\in \{1,\dots, l\}$, be connected normal semi-simple subgroups
  such that $G=\prod_i G_i$, $G_i\cap G_j=\{e\}$ if $i\neq j$, and let
  $\Gamma_i=\Gamma\cap G_i$ be an irreducible lattice in $G_i$, for
  each $i\in \{1,\dots, l\}$, with $\Gamma_0=\prod_i \Gamma_i$ of
  finite index in $\Gamma$.  Observe that
  $G/\Gamma_0\sim\prod_iG_i/\Gamma_i$. Let $p:G/\Gamma_0\to G/\Gamma$
  be the finite-to-one covering and let $p_i:G/\Gamma_0\to
  G_i/\Gamma_i$ be the projections onto the factors.  Let
  $\flow^{(0)}_\R$ be the flow induced by the one-parameter group
  $\flow_\R$ on $G/\Gamma_0$ and let $\flow^{(i)}_\R$ be the projected
  flow on $G_i/\Gamma_i$, for all $i\in \{1,\dots, l\}$. Since
  $\Gamma_i$ is an irreducible lattice in $G_i$, whenever
  $\flow^{(i)}_\R$ is partially hyperbolic we can apply
  Corollary~\ref{cor:kleinbock-margulis}. Since $\flow^{(0)}_\R$ is
  partially hyperbolic there is at least one $j \in \{1, \dots, l\}$
  such that $\flow^{(j)}_\R$ is partially hyperbolic. By
  Corollary~\ref{cor:kleinbock-margulis}, the flow $\flow^{(j)}_\R$
  has a countable family $\{K_n \vert n\in \N\}$ of distinct minimal
  subsets of $G_j/\Gamma_j$ such that each $K_n$ supports an invariant
  probability measure $\eta_n$. For all $n\in \N$, let us define $\hat
  \mu_n:=\eta_n\times Leb$ on $G/\Gamma_0$. By construction the
  measures $\hat \mu_n$ are invariant, for all $n\in \N$, and have
  mutually disjoint supports.  Finally, since the map $p:G/\Gamma_0\to
  G/\Gamma$ is finite-to-one, it follows that the family of sets
  $\{p(K_n\times\prod_{i\neq j}G_i/\Gamma_i) \vert n\in \N\}$ consists
  of countably many disjoint closed sets supporting invariant measures
  $ \mu_n:=p_*\hat\mu_n$. The proof of the Proposition is therefore
  complete.
\end{proof}

%%% Local Variables: 
%%% mode: latex
%%% TeX-master: "InvDist_affine"
%%% End: 

\def\altdiff{{\torusdiff_0}}
\def\torusdiff{F}
\def\transl{\tau}
\def\vectfield{\chi}
\def\End{\operatorname{End}}
\def\semiss{\mathfrak{s}} \def\unipt{\mathfrak{v}} 
\def\altmap{{\transl_0}}

\section{The general case}
\label{sec:6}

We may now prove our main theorem. We shall consider separately the
two cases of a flow and of an affine map.

\begin{proof}[Proof of Theorem~\ref{thm:Main}]
  By Remark~\ref{rem:3:1} we may suppose that the flow
  $(G/D,\flow_\R)$ is ergodic. Let also assume that $G$ is simply
  connected, by possibly pulling back $D$ to the universal cover of
  $G$. Recall that by Theorem~\ref{thm:3:2} the ergodic flow
  $(G/D,\flow_\R)$ projects onto the ergodic flow $\big(L/{q(D)},
  q(\flow_\R)\big)$, where $L$ is the Levi factor of $G$ and $q:G\to
  L$ the projection of $G$ onto this factor. Assume that the finite
  measure space $\big(L/{q(D)}\big)$ is not trivial. Then the
  statement of the theorem follows from Proposition~\ref{prop:4:2} if
  the flow $\big(L/{q(D)}, q(\flow_\R)\big)$ is quasi-unipotent and by
  Proposition~\ref{prop:5:2} if it is partially hyperbolic.

  If the finite measure space $\big(L/{q(D)}\big)$ is reduced to a
  point, then, using again Theorem~\ref{thm:3:2}, we have $G/D \approx
  R/R\cap D$, where $R$ is the radical of $G$. We obtain in this way
  that our original flow is diffeomorphic to an ergodic flow on a
  finite volume solvmanifold. By Mostow's theorem (see
  Theorem~\ref{thm:mostow}), a finite volume solvmanifold is
  compact. Hence the statement of the theorem follows from
  Theorem~\ref{expandingminimal} if the projected flow is partially
  hyperbolic and by Proposition~\ref{prop:4:4} if it is
  quasi-unipotent. The proof is therefore complete.
\end{proof}

\begin{proof}[Proof of Theorem~\ref{thm:Main_Affine}]
  In the course of the proof we shall use many times the
  Lemmata~\ref{lem:3homogflow:1} and~\ref{lem:3:3} which allow us to
  say that, whenever a quotient map or a suspension of an affine
  diffeomorphism has an infinite dimensional space of invariant
  distribution of a given order, so does the affine
  diffeomorphism. The same statement applies to measures. In the
  sequel, the term ``by standard arguments'' will refer to the
  application of this line of reasoning to infer that an affine
  diffeomorphism has an infinite dimensional space of invariant
  distribution (or measures).

  Let $\map_0=g A$ be an affine map of $G/D$. Let~$\Ggroup= \Hgroup
  =\Aut(G)\approx\Aut(\mathfrak g)$ and let~$\Hgroup$ act on~$\Ggroup$
  by inner automorphims (i.e.\ by conjugation). The groups~$\Ggroup$
  and~$\Hgroup$ are real algebraic groups and~$\Hgroup$ acts
  rationally on~$\mathcal A$. Let~$\Gsubgp= \Ad_G(G)< \Ggroup$.  Since
  for all $A\in \Ggroup$ and all $x\in G$ we have $A\circ
  \Ad_G(x)\circ A^{-1}= \Ad_G(Ax)$, the group~$\Gsubgp$ is normal
  in~$\Ggroup$ and stable under the action of~$\mathcal H$
  on~$\mathcal G$

  The epimorphism $\Ad_G\colon G\mapsto \Gsubgp \approx G/Z(G)$, maps
  the closed subgroup $D$ to the subgroup $ DZ(G)$. Let
  $\Dlattice=\overline {DZ(G)}$. Then, the map $\Ad_G$ induces a
  smooth submersion of $G/D$ onto the finite volume space
  $\Gsubgp/\Dlattice\approx G/\overline {DZ(G)}$.  This submersion
  intertwines the affine map $\map_0=gA $ of $G/D$ with the affine map
  $ \map=\Ad_G(g)\circ \operatorname{Int}_{ \Aut(G)}(A)\in
  \Aff_\Hgroup(\Ggroup)$, with $ \operatorname{Int}_{ \Aut(G)}(A)\in
  \Hgroup$ the conjugation by $A\in \Aut(G)$.

  By standard arguments, if the space of $\map$-invariant measures on
  $\Gsubgp/\Dlattice$ (respectively, the space of $\map$-invariant
  distributions on $\Gsubgp/\Dlattice$ of a given order) has infinite
  dimension, so does the space of $\map_0$-invariant measures on $G/D$
  (respectively, the space of $\map_0$-invariant distributions on
  $G/D$ of the same order).

  By Lemma~\ref{lemma:reduct_transl}, there exist a connected Lie
  group $\Gphi$ such that $\Gsubgp < \Gphi$, a quasi-lattice
  $\biglattice$ of $\Gphi$ containing $\Dlattice$, a non trivial
  one-parameter subgroup $\flow_\R \subset \Gphi$ and a finite cyclic
  group $\finitegp$ of affine diffeomorphism of $\Gphi/\biglattice$
  with the following properties:
  \begin{enumerate}
  \item the group $\finitegp$ commutes with the flow of the
    one-parameter group $\flow_\R$ by left translations on $\Gphi$;
    hence we obtain a quotient flow $(\finitegp\backslash \Gphi/
    \biglattice, \flow_\R)$ on the double coset space
    $\finitegp\backslash\Gphi/\biglattice$.
  \item the flow $(\finitegp\backslash \Gphi/\biglattice, \flow_\R)$
    is smoothly conjugate to a suspension of the affine map $\map$ on
    $\Gsubgp/\Dlattice$.
  \end{enumerate}
  
  The structure of the proof is now analogous to the proof of
  Theorem~\ref{thm:Main} and proceeds by analyzing the different cases
  for the dynamics of the flow $\flow_\R$ on the homogeneous space $
  \Gphi/\biglattice$ and then deriving the consequences for the affine
  map $\map$ of $\Gsubgp/\Dlattice$ and finally for the original map
  $\map_0$ of $G/D$.

  First we notice that we may assume that the flow
  $(\Gphi/\biglattice, \flow_\R)$ is ergodic: in fact the ergodic
  decomposition of this flow yields an infinite dimensional space of
  $\flow_\R$-invariant signed measures, which may be averaged under
  the action of $\finitegp$, yielding again an infinite dimensional
  space of $\flow_\R$-invariant and $\finitegp$-invariant signed
  measures; then we conclude, by standard arguments that the affine
  map $\map_0$ preserves infinitely many ergodic mutually singular
  invariant probability measures. Thus we assume that the flow
  $(\Gphi/\biglattice, \flow_\R)$ is ergodic.

  By Theorem~\ref{thm:3:2}, the ergodic flow
  $(\Gphi/\biglattice,\flow_\R)$ projects onto the ergodic flow
  $\big(\Levi/q(\biglattice), q(\flow_\R)\big)$, where $\Levi$ is the
  Levi factor of $\Gphi$ and $q\colon\Gphi\to \mathcal L$ the
  projection of $\Gphi$ onto this factor. The group of affine
  diffeomorphisms $\finitegp$ of $\Gphi/\biglattice$ projects under
  $q$ to a quotient finite cyclic group of affine diffeomorphisms
  $\finitegp\circ q$ of the quotient space $\Levi/q(\biglattice)$
  conmmuting with the ergodic flow $q(\flow_\R)$.

  \smallskip \emph{Case A: Non trivial Levi factor $\Levi$.}  Suppose
  that the finite measure space $\mathcal L/q(\biglattice)$ is not
  trivial. We distinguish two cases according to whether the flow
  $\big(\Levi/q(\widehat{\mathcal D}), q(\flow_\R)\big)$ is partially
  hyperbolic or quasi-unipotent.

  \smallskip \emph{Partially hyperbolic flow on the Levi factor.} By
  Proposition~\ref{prop:5:2}, the flow $q(\flow_\R)$ has infinitely
  many distinct compact invariant sets supporting infinitely many
  invariant and mutually singular ergodic probability measure.  Then,
  by \emph{standard arguments}, the covering flow $(\widehat{\mathcal
    N}/\biglattice, \flow_\R)$ has infinitely many
  $\flow_\R$-invariant and mutually singular ergodic probability
  measures $(\mu_i)_{i\in I}$. By averaging this collection of
  measures under the action of the finite cyclic group $\finitegp$ we
  obtain an infinite sub-collection of probability measures on the
  quotient space $\finitegp \backslash\widehat{\mathcal
    N}/\biglattice$ which are invariant and ergodic for the quotient
  flow $\flow_\R$. Then, by \emph{standard arguments}, we conclude
  that the affine diffeomorphism $\map$, first, and the affine
  diffeomorphism $\map_0$, next, have an infinite set of probability
  invariant measures.

  \smallskip \emph{Quasi unipotent flow on the Levi factor.}  If the
  flow $\big(\Levi/q(\biglattice), q(\flow_\R)\big)$ is
  quasi-unipotent, by Proposition~\ref{prop:4:2}, there exists
  infinitely many independent $q(\flow_\R)$-inva\-riant distributions on
  the space $\Levi/q(\biglattice)$ of Sobolev order $1/2$ which are
  also invariant under the finite cyclic group $\finitegp\circ q$. By
  standard arguments, we obtain that the dimension of the space of
  distributions of Sobolev order $1/2$ on $\Gphi/\biglattice$ which
  are simultaneously $\flow_\R$-invariant and invariant under
  $\finitegp$ is infinite. This is the same as saying that the space
  of distributions of Sobolev order $1/2$ on $\finitegp\backslash
  \Gphi/\biglattice$ which are invariant under the flow $\flow_\R$ is
  infinite. Then, by \emph{standard arguments}, we conclude that the
  affine diffeomorphism $\map$, first, and the affine diffeomorphism
  $\map_0$, next, have an infinite dimensional space of invariant
  distributions of Sobolev order~$1/2$.

  \smallskip \emph{Case B: Solvable $\Gphi$.}  Thus the theorem is
  proved if the finite measure space $\Levi/q(\biglattice)$ is not
  trivial.  In the opposite case, by Theorem~\ref{thm:3:2} we may and
  will assume that $\Gphi/\biglattice$ is a finite volume
  solvmanifold.

  Then the same is true of the manifold $\Gsubgp/\Dlattice$, since
  $\Gsubgp$ is a subgroup of $\Gphi$ hence solvable, and of the
  manifold $G/D$, since $\Gsubgp\approx G/Z(G)$. By Mostow's
  Theorem~\ref{thm:mostow}, a finite volume solvmanifold is compact.

  If the flow $(\Gphi/\biglattice, \flow_\R)$ is partially hyperbolic,
  so are the maps $\map$ and $\map_0$. Hence, in this case by
  Theorem~\ref{expandingminimal}, the map $\map_0$ admits infinitely
  many minimal sets and independent invariant ergodic probability
  measures.

  \def\diffeo{h} If, on the contrary, the flow $(\Gphi/\biglattice,
  \flow_\R)$ is quasi-unipotent then, by
  Proposition~\ref{prop:quasi-unip4:1} there exists a diffeomorphism
  $\tilde{\diffeo}\colon \Gphi \to N$ onto a nilpotent Lie group
  $\nilpt$. The diffeomorphism $\tilde{\diffeo}$ satisfies the
  following properties:
  \begin{enumerate}
  \item It induces a quotient diffeomorphism $\diffeo\colon
    \Gphi/\biglattice \to \nilpt/\Delta$ onto a compact nilmanifold
    $\nilpt/\Delta$ conjugating the flow $ \flow_\R$ on
    $\Gphi/\biglattice $ with a nilflow $u_\R$ on~$\nilpt/\Delta$.
  \item Restricted to $\biglattice$, the diffeomorphism
    $\tilde{\diffeo}$ is a group isomorphism of $\biglattice$
    onto~$\Delta$. Thus we may identify $\biglattice\approx \Delta$
  \item It conjugates the finite cyclic group $\finitegp$ to a finite
    cyclic group of affine diffeomorphisms $\finitegp'$ commuting with
    the flow~$u_\R$.
  \end{enumerate}
  It follows, in particular, that $\biglattice$ is a (co-compact)
  lattice in $\Gphi$.

  By Lemma~\ref{lemma:finite_order_transl}, the group $\finitegp'$ is
  generated by a translation by an element of the center of
  $\nilpt$. This fact has several consequences. First, the group
  $\finitegp'\Delta$ is discrete and the quotient space
  $\finitegp'\backslash \nilpt/\Delta$ coincides with the compact
  nilmanifold $\nilpt/\finitegp'\Delta$.  Second, as the group
  $\finitegp'$ operates without fixed points on $\nilpt/\Delta$, so
  does the group $\finitegp$ on $\Gphi/\widehat{\mathcal D} $. Thus
  the quotient space $\finitegp\backslash\Gphi/\biglattice$ is a
  smooth manifold\footnote{This is actually true, by construction of
    $\finitegp$, in a more general situation, e.g.\ whenever we may
    assume that the flow $\flow_\R$ commuting with $\finitegp$ is
    minimal.}.  Third, the group $\finitegp$ commutes with the lattice
  $\biglattice$ and the diffeomorphism $\tilde{\diffeo}$ is an
  isomorphism of the group $\finitegp\biglattice$ onto the group
  $\finitegp'\Delta$. In particular we may regard the
  diffeomorphism~$\diffeo$ as a diffeomorphism of the double coset
  quotient space $\finitegp\backslash\Gphi/\biglattice$ onto $
  \nilpt/\finitegp'\Delta$, conjugating the flow $\flow_\R$ with the
  flow $u_\R$.

  \smallskip \emph{Suppose $\nilpt$ is not Abelian}. If the connected
  nilpotent group $N$ is not Abelian, then, by the results of
  \cite{FF}, the flow $u_\R$ on the compact nilmanifold $
  \nilpt/\finitegp'\Delta$ admits infinitely many independent
  invariant distributions of Sobolev order $1/2$. It follows that the
  same is true for the flow $\big(\finitegp\backslash
  \Gphi/\biglattice, \flow_\R\big)$ and by standard arguments for the
  affine map $\map_0$ on $G/D$.

  \smallskip \emph{Suppose $\nilpt$ is Abelian}. We are left to
  consider the case where the connected nilpotent group $\nilpt$ is
  Abelian. Then the compact nilmanifolds $\nilpt/\Delta$ and
  $\nilpt/\finitegp'\Delta$ are tori. The flow of the one-parameter
  group $u_\R$ on these manifolds is conjugate, respectively, to the
  flows $(\Gphi / \biglattice, \flow_\R)$ and $(\finitegp \backslash
  \Gphi / \biglattice, \flow_\R)$.  In particular the manifolds $
  \Gphi / \biglattice$ and $\finitegp \backslash \Gphi / \biglattice$
  are tori, up to a diffeomorphisms.

  Recall that the flow $(\finitegp \backslash \Gphi/ \biglattice,
  \flow_\R)$ is the suspension of the affine diffeomorphism
  $(\Gsubgp/\Dlattice,\map)$. Thus the manifold
  $M:=\diffeo(\Gsubgp/\Dlattice)$ is a submanifold of the torus $
  \nilpt/\finitegp'\Delta$, such that the linear flow $u_\R$ has
  constant return time to $M$.  By Lemma~\ref{lem:4quasi-unip:1} the
  manifold $M$ is a subtorus $\T^k$ of $ N/\finitegp'\Delta$ and the
  return map to $M=\T^k$ is a translation on this torus. Thus
  $(\Gsubgp/\Dlattice,\map)$ is diffeomorphically conjugate to a torus
  translation $(\T^k, \tau)$.
  
  As $\biglattice$ is a lattice and $\Gphi/\biglattice$ a compact
  torus (up to a diffeomorphism), by construction of the groups
  $\Gphi$ and $\biglattice$, we have that the subgroup
  $\Dlattice<\biglattice$ is a lattice in the subgroup
  $\Gsubgp<\Gphi$.

  By definition we have $\Gsubgp\approx G/Z(G)$ and $\Dlattice
  =\overline{DZ(G)} $. However, as $\Dlattice$ is, by construction, a
  subgroup of the discrete group $\biglattice$, a lattice in
  ${\Gsubgp}$ we obtain that $\Dlattice =DZ(G) $, that is $DZ(G) $ is
  a closed subgroup of the solvable group~$G$. Thus, in the present
  situation, we have an affine diffeomorphism $\map_0$ of the compact
  solvmanifold $G/D$ inducing, via the submersion $\Ad_G\colon G/D\to
  G/DZ(G)$, an affine map $\map$ of the quotient space $G/DZ(G)$ which
  is conjugated by a diffeomorphism $\torusdiff\colon G/DZ(G)\to \T^d$
  to an ergodic translation of the torus $\T^d$.

  The minimal sets of $\map_0$ are compact subsets of $G/D$ surjecting
  onto $G/DZ(G)$, hence carrying a unique $\map_0$-invariant
  measure. Since $G/D$ is connected, either $\map_0$ is a minimal
  diffeomorphism of $G/D$ or there are infinitely may disjoint minimal
  sets of $\map_0$, and in particular infinitely $\map_0$-invariant
  independent probability measure on $G/D$. Thus we may assume that
  the diffeomorphism $\map_0$ is (uniquely) ergodic. This implies by
  Theorem~\ref{expandingminimal} that the diffeomorphism $\map_0$ is
  quasi-unipotent.

  By Corollary~\ref{cor:6gener-case:1} below, the diffeomorphism
  $\map_0$ either admits infinitely many invariant independent
  distributions or is smoothly diffeomorphic to a translation on a
  torus.

  The proof is therefore complete.
\end{proof}

For any manifold $M$ let $\vectfield(M)$ denote the space of vector
fields on $M$.  In what follows we identify the Lie algebra $\mathfrak
g$ of any Lie group $G$ with the space $\vectfield(M)^G$ of right
invariant vector fields on $G$. For any homogeneous space $G/D$, the
Lie algebra $\mathfrak g$ is identified with the subspace
$\vectfield(G/D)^G \subset \vectfield(G/D)$ given by the projections
on $G/D$ of the right invariant vector fields on $G$.

\begin{lemma}
  \label{lem:6gener-case:1}
  Let $G$ be a connected, simply connected, solvable Lie group, let
  $D<G$ be a lattice in $G$ and let $\map_0=u_0A_0$ be an ergodic
  affine quasi-unipotent diffeomorphism of~$G/D$. Assume that $Z(G)D$
  is a closed subgroup of $G$.  Let $\torusdiff\colon G/Z(G)D\to \T^d$
  be a diffeomorphism conjugating the map $\map$, induced by $\map_0$
  on $G/Z(G)D$, with an ergodic translation $\transl$ of $\T^d$.
  \begin{equation*}
     \iftikz
    \begin{tikzcd}
      G/D \ar{d} \ar{r}{\map_0}
      & G/D \ar{d}\\
      G/Z(G)D \ar{r}{\map} \ar{d}{\torusdiff}
      & G/Z(G)D \ar{d}{\torusdiff}\\
      \T^d \ar{r}{\transl} & \T^d
    \end{tikzcd}
    \else
    \text{Qui c'era une diagramma: aggiorna il tuo TeX!}
    \fi
  \end{equation*}
  Then there exists a structure of nilmanifold on $G/D$ (of degree of
  nilpotency at most $2$) with respect to which the map~$\map_0$ is
  affine and unipotent. More precisely, there exists a connected,
  simply connected, nilpotent Lie group $\nilpt$, a lattice $\Gamma <
  \nilpt$, a diffeomorphism $\altdiff\colon G/D \to\nilpt/ \Gamma $
  and an unipotent ergodic affine diffeomorphism $\transl_0\colon
  \nilpt/ \Gamma \to \nilpt/ \Gamma$ such that $ \altdiff \circ \map_0
  = \transl_0 \circ \altdiff$.
\end{lemma}
\begin{proof}
  \emph{Step 1.} Let $Z_0(G)$ the connected component of the identity
  of the center $Z(G)$ of $G$. The covering $G/Z(G)_0D\to G/Z(G)D$ is
  finite since $G/D$ is a compact solvmanifold, by Mostow's Theorem,
  which covers $G/Z(G)_0D$. Hence $G/Z(G)_0D$ is diffeomorphic to a
  torus. Thus, with no loss of generality, we shall assume that $Z(G)$
  is connected.

  Let $Z= Z(G)/(Z(G)\cap D)$. The group $Z$ acts freely on $G/D$ by
  left translations. Since the orbit space $G/Z(G)D$ is Hausdorff, and
  $G/D$ is compact, the orbits of $Z$ are compact. Hence $Z$ is a
  compact connected Abelian Lie group acting freely on
  $G/D$. Composing the projection $\pi\colon G/D\to G/Z(G)D$ with the
  diffeomorphism $\torusdiff\colon G/Z(G)D\to \T^d$, the map $p\colon
  G/D \to \T^d$ so obtained endows the space $G/D$ with the structure
  of a principal $Z$-bundle over $\T^d$.
  \begin{equation*}
    \iftikz
    \begin{tikzcd}
      G/D \ar{r}{\map_0} \ar{d}{p}
      &  G/D \ar{d}{p}\\
      \T^d \ar{r}{\transl} & \T^d
    \end{tikzcd}
    \fi
  \end{equation*}

  Let fix a connection $\omega$ for the principal bundle $(G/D, p)$.

  Denote by $\mathfrak z\approx \R^s$ the Lie algebra of the group
  $Z$.  The Lie algebra $\mathfrak z$ may be identified to the
  \emph{fundamental vertical vector fields} on $G/D$, generators of
  the left action of $Z(G)$ on $G/D$.

  For any ``constant'' (i.e.\ invariant under group translation)
  vector field $ X$ on $\T^d$ let $X^*$ be its horizontal lift for the
  connection $\omega$. We call such lifted vector fields $X^*$, the
  \emph{fundamental horizontal vector fields} (for the connection
  $\omega$).  Let $X_1^*, \dots, X_d^*$ be a basis of fundamental
  horizontal vector fields for the connection $\omega$ projecting to
  constant fields $ X_1, \dots, X_d$.  As the $X_i$'s commute, we have
  \[
  [X_i^*, X_j^*] = - \Omega(X_i^*,X^*_j)
  \]
  where $\Omega$ is the $\mathfrak z$-valued curvature $2$-form of the
  connection~$\omega$. We recall that, since~$Z$ is Abelian, the
  curvature form of any connection on $M$ is simply the differential
  of the $\mathfrak z$-valued connection form.

  Let $\widetilde \Omega$ the $\mathfrak z$-valued $2$-form on $\T^d$
  defined by $\widetilde \Omega( X_i, X_j) =\Omega(X_i^*,X^*_j)$, so
  that $\Omega= p^* \widetilde \Omega$.  The $2$-form $\widetilde
  \Omega$ is closed (since $p^*$ is injective and $\D \Omega = \D^2
  \omega=0$). Thus $ \widetilde \Omega$ is cohomologous to a constant
  $\mathfrak z$-valued $2$-form $\Omega_0$ on $\T^d$, that is, there
  is a $\mathfrak z$-valued $1$-form $\lambda$ on $\T^d$ such that $
  \widetilde\Omega = \Omega_0 +\D \lambda$. Let us define a new
  connection by $\omega'= \omega +p^*\lambda$. If $X_1', \dots, X_d'$
  are the horizontal lifts of the constant fields $ X_i$'s for the
  connection~$\omega'$, since $p_*X_i'=p_*X_i^*= X_i$ for all
  $i=1,\dots d$, we have
  \[
  \begin{split}
    [X_i', X_j'] &= - \Omega'(X'_i,X'_j) = - \D \omega' (X'_i,X'_j)\\
    &= -\D\omega (X'_i,X'_j) +\D p^*\lambda (X'_i,X'_j) \\
    &= - \Omega (X^*_i,X^*_j)+ p^*\D\lambda (X^*_i,X^*_j)\\
    &= - \widetilde \Omega ( X_i, X_j)+ \D\lambda ( X_i, X_j)=
    -\Omega_0( X_i, X_j)
  \end{split}
  \]
  Let $(V^*_\alpha)$ be a basis of fundamental vertical vector fields
  associated to a basis of $\mathfrak z$ denoted by the same
  letters. Then for all $\alpha=1,\dots$, as the the group $Z$ is
  Abelian, we have
  \[ [V^*_\alpha, X_j']=0.
  \]
  We conclude that the vector fields $X_1', \dots, X_d'$ and
  $(V^*_\alpha)$ on $G/D$ generate a $(d+s)$-dimensional nilpotent Lie
  algebra $\mathfrak n$ of degree of nilpotency $2$ at most. Let $N$
  be the simply connected, connected nilpotent group of Lie algebra
  $\mathfrak n$. The group $N$ operates locally faithfully on $G/D$
  via an action $\alpha\colon N \times G/D\to G/D$, whose generators
  are the vector fields $X_1', \dots, X_d'$ and $(V^*_\alpha)$.  As
  the sub-algebra $\mathfrak z$ is contained in~$\mathfrak n$, the
  universal cover $\widetilde Z$ of the group $Z$ is contained in
  $N$. The group $N/\tilde Z$ is isomorphic to~$\R^d$ via a mapping
  sending the generators $X'_i+\mathfrak z$ to standard generators
  of~$\R^d$. Since we have covering homomorphisms $\widetilde Z\to
  Z(G)\to Z$, the $\tilde Z$-orbits on $G/D$ coincide with the
  $Z(G)$-orbits, in particular they are closed.  It follows that the
  action $\alpha$ of $N$ on $G/D$ induces a quotient action $\bar
  \alpha \colon N/\tilde Z \times G/Z(G)D \to G/Z(G)D$ of the Abelian
  group $N/\tilde Z\approx \R^d$ on $ G/Z(G)D $. It is plain that the
  composition with the diffeomorphism $\torusdiff \colon G/Z(G)D \to
  \T^d$ yields an action of $N/\tilde Z\approx \R^d$ on $\T^d$ which
  is simply the action of $\R^d$ on $\T^d$ with generators $X_1$,
  \dots, $X_d$, i.e.\ the plain action $\R^d$ on $\T^d$ by
  translations. It follows that the action of $N$ (or $N/\tilde Z$) on
  the compact space $ G/Z(G)D $ is transitive.  Since the $\widetilde
  Z$-orbits are compact, we conclude that the action on $N$ on $G/D$
  is transitive.

  Fixing a point $x_0\in G/D$ and defining $\Gamma$ as the isotropy
  group $\{ n\in N \mid \alpha(n,x_0)=x_0\}$ of the point $x_0$, we
  obtain a diffeomorphism $ N/\Gamma \to G/D$ whose inverse will be
  denoted $\altdiff\colon G/D\to N/\Gamma $. We leave to the reader
  the easy verification that the induced quotient map $G/Z(G)D\to
  N/\tilde Z \Gamma$ coincides with the given diffeomorphism
  $\torusdiff$, via the identification $N/\tilde Z \approx \R^d$
  defined above. In particular the map $ G/D\to N/\tilde Z \Gamma
  =\T^d$ coincides with the principal bundle projection~$p$. We
  summarize the above construction with the following diagram:
  \begin{equation*}
    \iftikz
    \begin{tikzcd}%[column sep=small]
      G/D \ar{r}{\altdiff} \ar{d}{\pi} \ar{dr}{p}
      &   N/\Gamma  \ar{d} \\
      G/Z(G)D \ar{r}{\torusdiff} & N/\tilde Z\Gamma\approx \T^d
    \end{tikzcd}
    \fi
  \end{equation*}
  with
  \begin{equation}
    \label{eq:6gener-case:11}
    z F_0(xD) = F_0(z xD),  \quad\forall z\in Z(G),\quad \forall xD\in G/D.
  \end{equation}
  Henceforth the nilmanifold $N/\Gamma$ will be endowed with the above
  defined connection~$\omega'$ having (constant) curvature
  $\Omega_0$. To simplify notations these forms will be renamed
  $\omega$ and $\Omega$. We recall that for any fundamental horizontal
  fields $X^*, Y^*$ projecting to constant fields $ X$ and $ Y$ and
  any fundamental vertical vector field $V^*$ we have $[X^*,Y^*]=
  -\Omega( X, Y)$ and $[X^*,V^*]=0$.  \smallskip

  \noindent\emph{Remark.} The construction above depends on the 
  arbitrary choice of a primitive $\lambda$ of the exact form $\tilde
  \Omega -\Omega_0$. Clearly $\lambda$ is determined up to a closed
  one-form, i.e.\ up to a form $ \lambda_0 + \D f$, with $ \lambda_0$
  and $f$, respectively, a constant one-form and a smooth function on
  the torus $\T^d$. The effect of adding a constant one-form
  to~$\lambda$ consists in composing the map $F_0$ with a
  diffeomorphism $N/\Gamma \to N/\Gamma'$ induced by an automorphism
  of $N$ which projects to the identity automorphism of $N/\tilde
  Z$. Adding an exact one-form $\D f$ to $ \lambda$ results into
  composing the diffeomorphism $F_0$ with the fiber-wise
  diffeomorphism $ x\Gamma \mapsto \big(\exp f(x\tilde Z\Gamma)\big)
  x\Gamma$.

  Thus the group structure of $N$ is only determined up to these
  ambiguities.

  \medskip

  \emph{Step 2.}  Recall that the Lie algebra $\mathfrak g$ is
  identified with the space $\vectfield(G/D)^G$ of vector fields
  generating the left action of $G$ on $G/D$.  The push-forward map of
  vector fields ${\map_0}_* = (\D \map_0)\circ {\map_0}^{-1}$ induced
  by the affine diffeomorphism $\map_0=u_0A_0$ maps the Lie algebra
  $\vectfield(G/D)^G\approx \mathfrak g$ onto itself and it is easily
  identified with the automorphism of $\mathfrak g$ defined by
  $B_0=\Ad_G(u_0)\circ A_0$. By hypothesis the auto\-morphism~$B_0$ is
  quasi-unipotent.

  The center $Z(G)$ is a characteristic subgroup of~$G$, hence any
  automorphism of~$G$ restricts to an automorphism of $Z(G)$.
  Furthermore, the restriction of the automorphism $B_0$ to the
  sub-algebra $\mathfrak z$ coincides with $A_0$. Since
  $A_0(Z(G))=Z(G)$ and $A_0(D)=D$ the automorphism $A_0$ defines a
  quasi-unipotent automorphism of the torus $Z(G)/Z(G)\cap D$. It
  follows that the spectrum of $B_0$ restricted to $\mathfrak z$
  consists of roots of the unity.  Equivalently the spectrum of
  ${\map_0}_*\,$, restricted to the space of fundamental vertical
  vector fields, consists of roots of the unity.

  \smallskip

  \emph{Step 3.} Let $\altmap = \altdiff\circ \map_0 \circ
  \altdiff^{-1}$. The map $\altmap\colon N/\Gamma \to N/\Gamma$
  induces an ergodic translation~$\transl$ on the quotient torus
  $N/\tilde Z \Gamma$:
  \begin{equation*}
    \iftikz
    \begin{tikzcd}%[column sep=small]
      G/D \ar{r}{\altdiff} \ar{d} \ar{d}{\map_0}\ar[bend left]{rr}{p}
      & N/\Gamma \ar{d}{\altmap} \ar{r} & N/\tilde Z\Gamma \approx
      \T^d \ar{d}{\transl} \\ G/D \ar{r}{\altdiff}\ar[bend
      right]{rr}{p} & N/\Gamma \ar{r} & N/\tilde Z\Gamma \approx \T^d
    \end{tikzcd}
    \fi
    \quad\begin{matrix}
      \transl \circ p = p \circ \map_0\\
      \transl \circ \torusdiff =  \torusdiff  \circ\map
    \end{matrix}
  \end{equation*}

  Since for any $z\in Z(G)$ and any $xD\in G/D$ we have $\map_0(z xD)
  = A_0(z) \map_0( xD)$, and since by
  formula~\eqref{eq:6gener-case:11} the diffeomorphism $\altdiff$
  intertwines the actions of $Z(G)$ on the spaces $G/D$ and $
  N/\Gamma$, we obtain a similar identity for the diffeomorphim
  $\altmap$:
  \begin{equation}
    \label{eq:6gener-case:9}
    \altmap(z x\Gamma) = A_0(z) \altmap( x\Gamma), \qquad \forall
    z\in Z(G), \quad \forall x\Gamma\in  N/\Gamma,
  \end{equation}
  or, equivalently,
  \begin{equation*}
    (\altmap)_* V^*  =  A_0(V^*),
  \end{equation*}
  for any fundamental vector field $V^*$ on $N/\Gamma$.

  By definition constant vector fields $X$ on $N/Z(G)D$ are vector
  fields invariant by all translations, hence satisfying
  \[
  (\transl)_* X = X\,.
  \]
  Thus, for any fundamental horizontal vector field $X^*$ projecting
  to a constant vector field $X$, we have
  \begin{equation}
    \label{eq:6gener-case:15}
    (\transl_0)_* X^*  = X^* + \mu(X),
  \end{equation}
  with $\mu$ a smooth $1$-form on $N/Z(G)D$ with values in $\mathfrak
  z$. From this identity it follows that, for any two fundamental
  horizontal vector fields $X_1^*$ and $X_2^*$ projecting to constant
  vector fields $X_1$ and $X_2$, we have
  \[
  (\transl_0)_* [X_1^*,X_2^*] = [ X_1^* + \mu(X_1),X_2^* + \mu(X_2)] =
  [X_1^*,X_2^*] + \D \mu(X_1, X_2).
  \]
  Using the identity~\eqref{eq:6gener-case:9} and considering that
  $[X_1^*,X_2^*]$ equals the fundamental vertical vector field
  $-\Omega(X_1, X_2)$ we obtain the identity
  \[
  \D \mu = \Omega - A_0\circ \Omega.
  \]
  As the right hand side is a constant $2$-form on $N/\tilde Z\Gamma$
  and the left hand side and exact $2$-form, both terms are
  zero. Hence
  \begin{enumerate}
  \item the one-form $\mu$ is closed, and
  \item $A_0(V^*) = V^*$ for all $V^*\in [\mathfrak n,\mathfrak n]$.
  \end{enumerate}

  \smallskip

  \emph{Step 4.} Having studied the spectrum of the automorphism
  $B_0=u_0A_0u^{-1}\in\Aut(\mathfrak g)$, restricted to fundamental
  vertical vector fields (i.e.\ to $\mathfrak z$), we now proceed to
  consider the spectrum of the automorphism $B\in \Aut(\mathfrak
  g/\mathfrak z)$ induced by $B_0$ on $\mathfrak g/\mathfrak z$. As by
  previous remarks, the automorphism~$B$ will be identified with the
  restriction of the push-forward map ${\map}_*$ to the vector fields
  arising from the left action of $G/Z(G)$ on~$G/Z(G)D$.

  The Lie algebras $\mathfrak g\approx \vectfield(G/D)^G$ and
  $\mathfrak g/\mathfrak z\approx \vectfield(G/Z(G)D)^{G/Z(G)}$ are
  mapped by the push-forward maps $\altdiff_*$ and $\torusdiff_*$
  respectively to isomorphic sub-algebras $\altdiff_* (\mathfrak g)
  \subset\vectfield(N/\Gamma) $ and $\torusdiff_* (\mathfrak
  g/\mathfrak z) \subset \vectfield(N/\tilde Z\Gamma)$ of vector
  fields on $N/\Gamma$ and $N/\tilde Z\Gamma$ according to the
  following diagram
  \begin{equation*}
    \iftikz
    \begin{tikzcd}%[column sep=small]
      Y'\in \mathfrak g \ar{r}{\altdiff_*} \ar{d}{\pi_*} \ar{dr}{p_*}
      & \bar Y\in \chi(N/\Gamma) \ar{d}
      \\
      \hat Y\in \mathfrak g/\mathfrak z
      \ar{r}{\torusdiff_*} %\ar{d} {\hat A}
      & Y\in \chi(N/\tilde Z\Gamma) %\ar{d} {(\transl)_*}
    \end{tikzcd}
    \fi.
  \end{equation*}
 
  \smallskip The automorphism $B\in \Aut(\mathfrak g/\mathfrak z)$ is
  conjugated by the map $\torusdiff_*$ to an automorphisms of the
  sub-algebras $\torusdiff_* (\mathfrak g/\mathfrak z)$, still denoted
  by $B$, and coinciding with the push forward map $\transl_*$:
  \begin{equation*}
    \iftikz
    \begin{tikzcd}
      \mathfrak g/\mathfrak z \ar{r}{\torusdiff_*}
      \ar{d}[swap]{{\map}_*= B} & \torusdiff_* (\mathfrak g/\mathfrak
      z)\subset \vectfield(N/\tilde Z\Gamma) \ar{d}{{\transl}_*= B}
      \\
      \mathfrak g/\mathfrak z \ar{r}{\torusdiff_*} & \torusdiff_*
      (\mathfrak g/\mathfrak z)\subset\vectfield(N/\tilde Z\Gamma)
    \end{tikzcd}
    \fi
  \end{equation*}
  Thus we have, for any vector field $Y\in \torusdiff_* (\mathfrak
  g/\mathfrak z)$,
  \begin{equation}
    \label{eq:6gener-case:13}
    \transl_* Y = B(Y).
  \end{equation}
  Let $y=(y_1,\dots, y_d)$ be ``linear coordinates'' on the torus
  $N/\tilde Z\Gamma$, for which the lattice $\tilde Z\Gamma$ is the
  lattice $\Z^d$ in $\R^d$ and the (ergodic) translation $\transl$ is
  the translation modulo $\Z^d$ by the irrational vector $\alpha\in
  \R^d$. In these coordinates the differential of the $\transl$ is the
  identity and we obtain, for all vector fields $Y$ on $N/\tilde
  Z\Gamma$ and all $y\in N/\tilde Z\Gamma$
  \[
  (\transl)_*Y ( y) = Y (\transl^{-1}(y)).
  \]
  Thus, by the identity \eqref{eq:6gener-case:13}, we obtain that for
  all $Y\in \torusdiff_*(\mathfrak g/\mathfrak z)$ we have
  \begin{equation}
    \label{eq:6gener-case:4}
    Y (\transl^{-1}( y)) = B ( Y)(y), \quad \text{\ for all\ }
    y\in N/\tilde Z\Gamma.
  \end{equation}

  Let $Y_1, \dots, Y_d$ be a basis of $\torusdiff_*(\mathfrak
  g/\mathfrak z)$ and let $X_1, \dots, X_d$ be a basis of constant
  vector fields on $N/\tilde Z\Gamma$. Then we can write
  \begin{equation}
    \label{eq:6gener-case:5}
    Y_i= \sum_{j} H_{ij} X_j, \quad \text {\ for all\ } i=1,\dots d.
  \end{equation}
  with $H =( H_{ij})\colon N/\tilde Z\Gamma \to
  \operatorname{Gl}(\R^d) $ a smooth function on $N/\tilde
  Z\Gamma$. Writing the matrix valued function $H$ in Fourier series
  with respect to the coordinates $(y_i)$, we have
  \begin{equation}
    \label{eq:6gener-case:6}
    H(y) =
    \sum_{n \in \Z^d} h_n e_n(y),
  \end{equation}
  with $e_n(y) = \exp(2\pi \imath n\cdot y)$.  Denoting the matrix on
  the automorphism $B$ with respect to the basis $( Y_i)$ by the same
  letter $B$, the equation~\eqref{eq:6gener-case:4} reads
  \[
  \big( B - e_{n}(-\alpha)I\big) h_n =0
  \]
  or
  \[
  h_n^T (B^T - e_n(- \alpha) I) =0.
  \]
  This identity shows that the matrix $h_n^T$ vanishes on the range of
  $( B^T - e_{n} (-\alpha) I) $. Thus, in Fourier series of $H(y)$,
  there are at most $d$ coefficients $h_n$ which do not vanish; they
  correspond to the indices $n$ for which $e_{n} (-\alpha) $ is an
  eigenvalues of the matrix $ B$. Let $E_{n_1}$, \dots , $E_{n_k}$ be
  the generalized eigenspaces of $ B^T$ corresponding to the
  eigenvalues $ e_{n_1} (-\alpha) $ \dots $e_{n_k} (-\alpha)$. Then
  $\oplus_{ \ell=1}^k E_{n_\ell} = \C^d$ and $\ker h^T_{n_\ell}
  \supset \oplus_{m\neq \ell} E_{n_m}$. Since the matrix $H (y)$ is
  invertible for all $y\in N/\tilde Z\Gamma$, the kernel of
  $h^T_{n_\ell}$ cannot be larger than $\oplus_{m\neq \ell} E_{n_m}$
  and consequently we have the identity $\ker h^T_{n_\ell} =
  \oplus_{m\neq \ell} E_{n_m}$. It follows that the matrix
  $h_{n_\ell}$ is a linear surjective map of $\C^d$ onto $E_{n_\ell}$
  and a linear isomorphism when restricted to $E_{n_\ell}$. Hence $ B|
  E_{n_\ell} = e_{n_\ell}(-\alpha) I$. We have proved that the
  automorphism $ B$ of $\torusdiff_*(\mathfrak g/\mathfrak z) $ (or
  equivalently the automorphism $B$ of $\mathfrak g/\mathfrak z$) is
  semi-simple with spectrum $ \{ e_{n_\ell}(-\alpha)| \ell=1,\dots,
  k\} $ corresponding to eigenspaces $ E_{n_\ell}\subset
  \torusdiff_*(\mathfrak g/\mathfrak z) $, $ \ell=1,\dots, k $.

  Since $\alpha$ is a irrational number, none of the numbers
  $e_{n_\ell}(-\alpha)$ occurring in the spectrum of $B$ is a root of
  unity, if $n_\ell\neq 0$. From the results of \emph{Step 2}, we
  deduce that every eigenspace $E_{n_\ell}\subset \mathfrak
  g/\mathfrak z$ with $n_\ell\neq 0$ lift to an eigenspace
  $E'_{n_\ell}$ of the automorphism~$B_0$ of $\mathfrak g$. The
  subspace $E_0\subset \mathfrak g/\mathfrak z$ is an Abelian
  sub-algebra of $\mathfrak g/\mathfrak z$, mapped to constant vector
  fields by the conjugation $\torusdiff_*$.

  Chasing definitions, we conclude that the identity
  \begin{equation}
    \label{eq:6gener-case:14}
    \altmap_* \bar Y =e _{n_\ell}(-\alpha) \bar Y .
  \end{equation}
  holds true for all $\bar Y\in \altdiff_*( E'_{n_\ell} )$.  \medskip

  \emph{Step 5.}  Let $(Y_\ell)_{\ell=1,\dots,d}$ be a basis of the
  sub-algebra $\torusdiff_*(\mathfrak g/\mathfrak z) $ of eigenvectors
  of the automorphism $B\approx \transl_*$ with eigenvalues $
  \lambda_i=e_{n_{\ell}}(-\alpha)$, $\ell=1,\dots,d$. The elements $(
  Y_\ell)$ with $n_{\ell}\not= 0$ come in conjugate pairs. We may
  assume that elements $(Y_\ell)$ with $n_{\ell}= 0$ are real. Let
  $(Y_\ell^*)_{\ell=1,\dots,d}$, be the horizontal lifts of the fields
  $(Y_\ell)_{\ell=1,\dots,d}$ (then this set of elements of $\mathfrak
  g$ is closed under conjugation).

  Let $I_0$ be the set of indices $\ell=1,\dots,d$ such that
  $n_\ell=0$, and let $I^c_0$ its complement in the integer interval
  $[1,d]$. For $\ell \in I^c_0$ let $\bar Y_\ell\in \altdiff_*(
  E'_{n_\ell})$ be an eigenvector $\altmap_*$ projecting to $Y_\ell$,
  so that
  \begin{equation}
    \label{eq:6gener-case:17}
    \altmap_* \bar Y_\ell=e _{n_\ell}(-\alpha) \bar Y_\ell ,\quad \forall
    \ell\in I^c_0;
  \end{equation}
  for $ \ell\in I_0$, let $\bar Y_\ell$ be a generalised eigenvector
  of eigenvalue $1$ for $\altmap_*$ projecting to to the (constant)
  field $Y_\ell$; then there exist vertical fields so that $V_{\ell,
    1}$, $V_{\ell, 2}$,\dots, $V_{\ell, j_\ell}$, such that
  \begin{equation}
    \label{eq:6gener-case:24}
    (\altmap)_*^k \bar Y_\ell= \bar Y_\ell + k V_{\ell, 1}+ \tbinom{k}{2} V_{\ell, 2}
    +\dots + \tbinom{k}{j_\ell} V_{\ell, j_\ell},\quad \forall
    \ell\in I^c_0\,\,\,\forall
    k\in \N.
  \end{equation}

  Choose fundamental horizontal vector fields
  $(X^*_\ell)_{\ell=1,\dots,d}$ so that at the point $\Gamma\in
  N/\Gamma$ they coincide with the horizontal vectors
  $(Y_\ell^*(\Gamma))_{\ell=1,\dots,d}$. Then the vector fields
  $(X^*_\ell)_{\ell=1,\dots,d}$ project to constant vector fields
  $(X_\ell)_{\ell=1,\dots,d}$ which at the point $\tilde Z\Gamma\in
  N/\tilde Z\Gamma$ coincide with the vectors $(Y_\ell(\tilde
  Z\Gamma))$, i.e.\ we have $ Y_\ell(\tilde Z\Gamma) =X_\ell $ , for
  all $\ell=1,\dots,d$.

  With these choices, the formula \eqref{eq:6gener-case:5}, in virtue
  of the definition \eqref{eq:6gener-case:6}, becomes
  \[
  Y_\ell= e_{n_\ell} X_\ell, \quad \ell=1,\dots,d,
  \]
  and, consequently, we have
  \[
  Y^*_\ell = e_{n_\ell} X^*_\ell \quad \ell=1,\dots,d.
  \]
  From the identity~\eqref{eq:6gener-case:15} we obtain
  \begin{equation}
    \label{eq:6gener-case:16}
    (\transl_0)_* Y_\ell^*  = e_{n_\ell}(-\alpha)\big(Y_\ell^* + 
    \mu(Y_\ell)\big),
  \end{equation}
  for all $\ell=1,\dots,d$.

  Let $\nu$ be the $\mathfrak z$-valued one-form on the torus
  $N/\tilde Z\Gamma$ defined by
  \begin{equation}
    \label{eq:6gener-case:27}
    \bar Y_\ell = Y^*_\ell + \nu(Y_\ell)%= Y^*_\ell +e_{n_\ell}  \nu(X_\ell)
    , \quad \forall \ell=1,\dots,d.
  \end{equation}
  For $Y=Y_\ell$, using \eqref{eq:6gener-case:14} and the fact that
  $(\transl_0)_* $ restricted to vertical vectors coincides with the
  automorphism $A_0$, we obtain the identity
  \[
  \begin{split}
    (\transl_0)_* (\nu (Y))(x)&= (\D \transl_0)_{
      \transl_0^{-1}(x)}\big(\nu_{
      \transl^{-1}(x)}(Y_{\transl^{-1}(x)} )\big)= A_0 \big(\nu_{
      \transl^{-1}(x)}(Y_{\transl^{-1}(x)})\big) \\ &= A_0 \Big(
    (\transl_*\nu)_x (\transl_* Y)_x\Big) =e _{n_\ell}(-\alpha) A_0
    \Big( (\transl_*\nu)_x (Y_x)\Big)\,..
  \end{split}
  \]
  where, as usual, the push-forward of a one form by a diffeomorphism
  is defined as the pull-back by the inverse diffeomorphism: for any
  scalar one-form $\beta$ on the torus $N/\tilde Z\Gamma$ we define
  $\tau_* \beta=(\tau^{-1})^* \beta$ (so that
  $\beta_y(X_\ell)=\beta_{\tau^{-1}y}(X_\ell)$ for all
  $\ell=1,\dots,d$). Hence, from~\eqref{eq:6gener-case:16} and the
  above definition~\eqref{eq:6gener-case:27}, we derive
  \begin{equation}
    \label{eq:6gener-case:21}
    (\transl_0)_* \bar Y_\ell   = e_{n_\ell}(-\alpha) \big[Y_\ell^*  +
    \mu(Y_\ell) +   (A_0\circ\tau_*\nu) (Y_\ell))\big], \quad 
    \forall \ell=1,\dots,d.
  \end{equation}
  On the other hand, for $\ell\in I^c_0$, the same
  definition~\eqref{eq:6gener-case:27} together with the
  identity~\eqref{eq:6gener-case:17} yields
  \[
  (\transl_0)_* \bar Y_\ell = e_{n_\ell}(-\alpha) \big[ Y^*_\ell +
  \nu(Y_\ell)], \quad \forall \ell\in I^c_0.
  \]
  Comparing the two expressions above we obtain, for all $\ell\in
  I^c_0$, the identity $\mu(Y_\ell) = \nu(Y_\ell) -(A_0\circ\tau_*\nu)
  (Y_\ell) $, or, equivalently,
  \begin{equation}
    \label{eq:6gener-case:18}
    \mu(X_\ell) = \nu(X_\ell) -(A_0\circ \tau_*\nu) (X_\ell) , \quad \forall \ell\in I^c_0.
  \end{equation}
  
  For the case $\ell\in I_0$, we consider the following
  generalizations of the formulas~\eqref{eq:6gener-case:16}
  and~\eqref{eq:6gener-case:21}: for all $\ell=1,\dots,d$, we have
  \begin{equation*}
    (\transl_0)^k_* Y_\ell^*  = e_{n_\ell}(-k\alpha)\big(Y_\ell^* +
    \sum_{j=1}^{k}( A_0^{j-1}\circ \tau^{j-1}_*\mu )(Y_\ell),
  \end{equation*}
  and
  \begin{equation*}
    (\transl_0)^k_* \bar Y_\ell   = e_{n_\ell}(-k\alpha) \big[Y_\ell^*
    +\sum_{j=1}^{k} (A_0^{j-1}\circ  \tau^{j-1}_*\mu)(Y_\ell) +  
    (A^k_0\circ \tau^k_*\nu)(Y_\ell) \big].
  \end{equation*}
  For $\ell\in I_0$, the definition~\eqref{eq:6gener-case:27} and the
  formula~\eqref{eq:6gener-case:24} give
  \[
  (\altmap)_*^k \bar Y_\ell= \bar Y_\ell^* + \nu(Y_\ell) + k V_{\ell,
    1}+ \tbinom{k}{2} V_{\ell, 2} +\dots + \tbinom{k}{j_\ell} V_{\ell,
    j_\ell},\quad \forall \ell\in I^c_0;
  \]
  Taking differences and considering that for $\ell\in I_0$ we have
  $Y_\ell=X_\ell$ we obtain
  \[
  \sum_{j=1}^{k} (A_0^{j-1}\circ \tau^{j-1}_*\mu)(X_\ell)= \nu(X_\ell)
  -(A^k_0\tau^k_*\nu)(X_\ell) + k V_{\ell, 1}+ \tbinom{k}{2} V_{\ell,
    2} +\dots + \tbinom{k}{j_\ell} V_{\ell, j_\ell}.
  \]
  and for $k=1$
  \begin{equation}
    \label{eq:6gener-case:20}
    \mu(X_\ell)= \nu(X_\ell)
    -(A_0\circ \tau_*\nu)(X_\ell) +V_{\ell, 1}
  \end{equation}

  \medskip

  \emph{Step 6.} Define a new connection $\omega_0$ on $N/\Gamma$ by
  setting
  \[
  \omega_0= \omega -\nu.
  \]
  The horizontal lifts of a constant vector field $X$ with respect to
  this connection is now given by the formula
  \begin{equation*}
    \tilde X^* = X^*+ \nu (X).
  \end{equation*}
  From the identity~\eqref{eq:6gener-case:15} we obtain
  \[
  (\tau_0)_* (\tilde X^*) = (\tau_0)_* (X^*) + (A_0 \circ\tau_*\nu)(X)
  = X^*+ \mu(X)+ (A_0 \circ\tau_*\nu)(X)
  \]
  Hence, for all $\ell\in I^c_0$, the
  identity~\eqref{eq:6gener-case:18}, $ \mu(X_\ell) = \nu(X_\ell)
  -(A_0\circ \tau_*\nu) (X_\ell)$, nous donne
  \[
  (\tau_0)_* (\tilde X_\ell^*) = X^*+
  \mu(X_\ell)+A_0\tau_*\nu(X_\ell)= X_\ell^* +\nu(X_\ell)= \tilde
  X_\ell^*
  \]
  For all $\ell\in I^c_0$, the identity~\eqref{eq:6gener-case:20},
  yields by a similar computation,
  \[
  (\tau_0)_* (\tilde X_\ell^*) = X_\ell^*+ V_{\ell,1}
  \]
  We have shown that the diffeomorphism $\tau_0$ has a constant
  jacobian matrix with respect to the basis of vector fields $\{\tilde
  X^*_i, V^{*}_\alpha\mid \ell=1,\dots,d, \alpha =1,\dots s\}$ and
  that $(\tau_0)_*(\tilde X^*_\ell) =\tilde X^*_\ell + \sum
  c_{\ell,\alpha} V^{*}_\alpha $.  \medskip

  \emph{Step 7.}  The $\mathfrak z$-valued one-form $\nu$ is $Z(G)$
  invariant, hence for any constant vector field $X$ on the torus
  $N/\tilde Z\Gamma$ and every fundamental vertical vector field
  $V^{*}$ we have $[V^{*}, \nu(X)]=0$ and therefore $[ V^{*},\tilde
  X^*]=[ V^{*}, X^*+ \nu (X)]=0$.  It follows that
  \[
  (\tau_0)_*[\tilde X^*_i, \tilde X^*_j]= [\tilde X^*_i, \tilde
  X^*_j].
  \]
  On the other hand we have
  \[ [\tilde X^*_i, \tilde X^*_j] = -\Omega ( X_i, X_j)+ \D \nu ( X_i,
  X_j)=: C( X_i, X_j).
  \]
  Hence
  \[
  (\tau_0)_*[\tilde X^*_i, \tilde X^*_j]= (A_0 \circ \tau_* C) (X_i,
  X_j) = C (X_i, X_j)
  \]
  that is,
  \[
  A_0 \circ \tau_* C= C.
  \]
  As all the eigenvalues of the automorphism $A_0$ are roots of unity
  by taking a suitable power of the automorphism $A_0$, we obtain the
  identity $A^k_0 \circ \tau^k_* C= C$, with $A^k_0$ unipotent. It
  follows, by the ergodicity of all the iterates of $ \tau$, that the
  two-form $C$ is constant, i.e.\ that $\D \nu =0$. We observe that
  the vector fields $\{\tilde X^*_i, V^*_\alpha\}$ form the basis of a
  nilpotent Lie algebra isomorphic to $\mathfrak n$, since the only
  non-trivial commutation relations are given by
  \[ [\tilde X^*_i, \tilde X^*_j] = -\Omega ( X_i, X_j).
  \]
  Summarizing, by the remark at the end of \emph{Step 1}, we may
  compose the diffeomorphism $F_0$ by a diffeomorphism $h$ of $N$
  projecting to the identity diffeomorphism of $N/\tilde Z\Gamma$, so
  that, after lifting the vector fields $X_i$, $\tilde X_i^*$ and
  $V^*_\alpha$ to $N$, we have $h_*X_i=\tilde X_i^*$ and $h_*
  V^*_\alpha =V^*_\alpha$. The vector fields $\{\tilde X^*_i,
  V^*_\alpha\}$ define a new Lie group structure on~$N$, with Lie
  algebra $\mathfrak n$, with respect to which the diffeomorphism
  $\transl_0$ has ``constant jacobian matrix'' since we have
  \[
  (\transl_0)_* V^*_\alpha = A_0(V^*_\alpha).
  \]
  and since there exist $V_i\in \mathfrak z$
  \[
  (\transl_0)_*\tilde X^*_i = \tilde X^*_i + V_i , \quad i=1,\dots d.
  \]

  Thus, up to replacing $F_0$ with $h\circ F_0$ and the lattice
  $\Gamma$ by $h(\Gamma)$ and renaming the $\tilde X^*_i$ as $X^*_i$,
  we have proved that there exists
  \begin{itemize}
  \item a connected, simply connected, nilpotent Lie group $\nilpt$,
  \item a lattice $\Gamma < \nilpt$,
  \item a diffeomorphism $\altdiff\colon G/D \to\nilpt/ \Gamma $,
  \item a diffeomorphism $\transl_0\colon \nilpt/ \Gamma \to \nilpt/
    \Gamma$
  \item and a basis $\{X^*_i, V^*_\alpha\mid i=1,\dots d,
    \alpha=1,\dots ,s\}$ of the Lie algebra $\mathfrak n$ of $N$
  \end{itemize}
  such that
  \begin{itemize}
  \item The only non-trivial commutation relations in $\mathfrak n$
    are given by
    \[ [ X^*_i, X^*_j] =V_{ij}\, \quad\text{with } V_{ij}\in \mathfrak
    z :=\operatorname{span}\{ V^*_\alpha\};
    \]
  \item The diffeomorphism $\altdiff$ intertwines the maps $\map_0$
    and $\transl_0 $:
    \[
    \altdiff \circ \map_0 = \transl_0 \circ \altdiff;
    \]
  \item The diffeomorphism $\transl_0$ has ``constant jacobian
    matrix'' given
    \[
    (\transl_0)_* V^*_\alpha = A_0(V^*_\alpha), \quad
    (\transl_0)_*\tilde X^*_i = \tilde X^*_i + V_i , \quad i=1,\dots
    d, \alpha=1,\dots ,s.
    \]
    with $A_0$ a quasi-unipotent automorphism of the Abelian
    subalgebra $\mathfrak z$ and $V_i\in \mathfrak z$.
  \end{itemize}
  It is immediate to deduce that $\tau_0$ is an affine map of
  $\nilpt$, hence concluding the proof. In fact let $B_0$ be the
  quasi-unipotent automorphism of $\mathfrak n$ defined by
  $(\transl_0)_*$. Let $\tilde\transl_0$ be the lift to $\nilpt$ of
  the diffeomorphim $\transl_0$ and set $\tilde \transl_0 (e_\nilpt)=
  x_0$. The map $\sigma\colon x\in \nilpt \mapsto x_0^{-1}
  \tilde\transl_0 (x)\in \nilpt$ fixes the neutral element
  $e_{\nilpt}\in \nilpt$. Its tangent map is the automorphism of
  $\mathfrak n$ given by $B_1:=\operatorname{Int}(x_0) \circ B_0$,
  which is easily verified to be quasi-unipotent. Thus the map
  $B_1^{-1}\circ \sigma$ is a diffomorphism of $N$ fixing the neutral
  element $e_{\nilpt}$ and whose tangent map induces the identity map
  $\mathfrak n$. It follows that $\tilde \transl_0(x) = x_0B_1(x)$ for
  all $x\in \nilpt$. Hence the map $ \transl_0$ is an affine
  quasi-unipotent diffeomorphism of $\nilpt/\Gamma$. As the
  diffeomorphism $\map_0$ is ergodic so it is the diffeomorphism $
  \transl_0$. By Corollary~1 of Parry's paper~\cite{MR0260975}, a
  quasi-unipotent ergodic affine diffeomorphism of a nilmanifold is
  unipotent. This concludes the proof.
\end{proof}

\begin{lemma}
  \label{lem:6gener-case:3}
  Let $\nilpt$ be a connected, simply connected, nilpotent Lie group,
  let $\Gamma < \nilpt$ be a lattice in $\nilpt$ and $\map\colon
  \nilpt/ \Gamma \to \nilpt/ \Gamma$ a $\Ad$-unipotent affine
  transformation, with $\map(x\Gamma)=u A(x)\Gamma$.  Then the
  automorphism $A$ is unipotent with a rational generalized Jordan
  basis.
\end{lemma}

By a rational basis, we mean a basis $\mathcal B$ of $\mathfrak n$
such that every linear form on $\mathfrak n$ which is integral on
$\log \Gamma$ takes rational values on every $v\in \mathcal B$. A
rational basis $\mathcal B$ is a generalized Jordan basis for the
automorphism $A\in \Aut(\nilpt)$ if its matrix with respect to this
basis is upper (or lower) triangular.
\begin{proof}
  The Lemma is immediate if $\nilpt$ is Abelian, since in this case
  $\D \map = A$ and $A(\Gamma)=\Gamma$.

  Suppose, by induction, that the Lemma is true for all nilmanifolds
  $\nilpt/ \Gamma$ such that the degree of nilpotency of $\nilpt$ is
  less than $n$. Let the degree of nilpotency of $\nilpt$ be equal
  to~$n$.

  Let $\{ \mathfrak n^{(k)}\}$ denote the descending central series of
  the Lie algebra $\mathfrak n$ of $N$, defined by induction by
  $\mathfrak n^{(0)}= \mathfrak n$ and $\mathfrak n^{(k+1)}=
  [\mathfrak n^{(k)}, \mathfrak n]$ for all $k\in \N$, and let
  $\nilpt^{(k)}$ be the corresponding analytic groups.

  Then $\nilpt/\nilpt^{(n-1)}$ is a nilpotent group of degree of
  nilpotency~$(n-1)$. Since the induced affine diffeomorphism $\bar
  \map$ of the nilmanifold $\nilpt/\Gamma\nilpt^{(n-1)}$ is unipotent,
  by the induction hypothesis, the induced automorphism $\bar A\in
  \Aut(\mathfrak n / \mathfrak n^{(n-1)})$ is unipotent and has a
  generalized rational Jordan basis, that is a basis \[\bar {\mathcal
    B}_{n-1}=\{\bar X_{1,1}, \dots, \bar X_{1,\ell_1}, \bar X_{2,1},
  \dots,\bar X_{2,\ell_2}, \dots, \bar X_{k,1}, \dots, \bar
  X_{k,\ell_k}\},\] such that \[\bar A \bar X_{i,j}= \bar X_{i,j} +
  \sum_{k=1}^{j-1} c_{i,k}\bar X_{i,k},\] with $c_{i,k}\in \Q$ for all
  possible values of $i$ and $ k$.

  Let $ \mathcal B_{n-1}=\{X_{1,1}, \dots, X_{1,\ell_1}, X_{2,1},
  \dots, X_{2,\ell_2}, \dots, X_{k,1}, \dots, X_{k,\ell_k}\}$ a system
  of rational vectors projecting to $\bar {\mathcal B}_{n-1}$ under
  the natural map $\mathfrak n \to \mathfrak n / \mathfrak n^{(n-1)}$.
  Since the differential $\D {\map}$ of the $\Ad$-unipotent affine
  transformation $\map$ preserves the center $\nilpt^{n-1}$ and it
  coincides with the automorphism~$A$, there exists a rational Jordan
  basis $\mathcal C_{n-1}$ of $\mathfrak n^{(n-1)}$ for the
  restriction of the automorphism~$A$ to $\nilpt^{n-1}$.

  It is now immediate to check that $\mathcal B_{n}=\mathcal
  B_{n-1}\cup \mathcal C_{n-1}$ is a generalized rational Jordan basis
  for the automorphism $A$, and that $A$ is unipotent.
\end{proof}

\begin{lemma}
  \label{lem:6gener-case:4}
  Let $\nilpt$ be a connected, simply connected, nilpotent Lie group,
  let $\Gamma < \nilpt$ be a lattice in $\nilpt$ and $\map\colon
  \nilpt/ \Gamma \to \nilpt/ \Gamma$ a unipotent ergodic affine
  transformation.  Then either the diffeomorphism~$\map$ admits
  infinitely many invariant independent distributions of Sobolev order
  $1/2$ or is smoothly diffeomorphic to a toral translation.
\end{lemma}
\begin{proof}
  Recall that $N$ is an algebraic affine group. The exponential and
  logarithm map $\exp\colon \mathfrak n \to \nilpt $ and $\log\colon
  \nilpt \to \mathfrak n $ are polynomial maps, hence rational. By
  Lem\-ma~\ref{lem:6gener-case:3}, if $\map(x\Gamma)=u A(x)\Gamma$, then
  the automorphism $A$ is unipotent with a rational Jordan basis. It
  follows that there exists a one-parameter group $B= \{B^t\}_{t\in
    \R}$ of automorphisms of $\nilpt $ such that $B^1=A$. With respect
  to a rational Jordan basis of $\mathfrak n$, the automorphisms
  $\{B^t\}$ are polynomial in the variable $t$, hence they operate
  rationally on $\nilpt $. It follows that the semi-direct product
  $H=\nilpt \rtimes B$ is an algebraic nilpotent group. The set
  $\{(u,A)^n\mid n\in \Z\}$ is a subgroup of~$H$; let $T$ be its
  Zariski closure in $H$. As $A$ is unipotent and $\nilpt $ connected
  and simply connected, the Abelian group $T$ is connected for the
  Hausdorff topology. It follows that the group $T$ is a one-parameter
  group of $H$, necessarily of the form $\{(u(t), B^t)\}_{t\in T}$,
  such that $(u(1), B^1) = (u,A)$.

  Define $\Lambda=\{(\gamma, B^n)\mid n\in \Z, \gamma\in \Gamma\}$,
  Since $B^1=A$ and $A(\Gamma)=\Gamma$, the set $\Lambda$ is a
  discrete subgroup of $H$. In fact the space $H/\Lambda$ is a compact
  nilmanifold. We consider the nilmanifold $\nilpt/\Gamma$ as a
  submanifold of $H/\Lambda$ via the embedding given by the mapping
  $x\Gamma \mapsto (x, 1) \Lambda$.

  Let~$(\flow_t)_{t\in \R}$ be the flow on $H/\Lambda$ given by left
  translation by the one-parameter group~$\{B^t\}_{t\in \R}$. It is
  immediate to check that the first return of the flow
  $(\flow_t)_{t\in \R}$ to the submanifold $\nilpt/\Gamma$ occurs at
  time $t=1$ and the first return map coincides with the affine
  map~$\map$. Hence the flow $(\flow_t)_{t\in \R}$ on the compact
  nilmanifold $H/\Lambda$ is the suspension of the affine
  diffeomorphism~$\map$ and in particular it is ergodic.

  By Theorem~\ref{thm:flaminio-forni}, the flow $(\flow_t)_{t\in \R}$
  either admits infinitely many invariant independent distributions of
  Sobolev order $1/2$ or is diffeomorphic to an ergodic translation
  flow on a torus. This latter possibility occurs only if the
  nilmanifold $H/\Lambda$ is a torus; this is the case only if the
  automorphism $A$ is trivial and the group $N$ is Abelian, that is if
  the affine map $\map$ is a translation and on a torus. Thus, unless
  the diffeomorphism~$\map$ is a translation on a torus, the
  diffeomorphism~$\map$ admits infinitely many invariant independent
  distributions of Sobolev order $1/2$.
\end{proof}

\begin{corollary}
  \label{cor:6gener-case:1}
  Let $G$ be a connected, simply connected, solvable Lie group, let
  $D<G$ be a lattice in $G$ and let $\map_0=u_0A_0$ be an ergodic
  affine quasi-unipotent diffeomorphism of~$G/D$. Assume that $Z(G)D$
  is a closed subgroup of $G$.  Let $\torusdiff\colon G/Z(G)D\to \T^d$
  be a diffeomorphism conjugating the map $\map$, induced by $\map_0$
  on $G/Z(G)D$, with an ergodic translation $\transl$ of $\T^d$.  Then
  either the diffeomorphism~$\map_0$ admits infinitely many invariant
  independent distributions of Sobolev order $1/2$ or is smoothly
  diffeomorphic to a toral translation.
\end{corollary}
\begin{proof}
  By Lemma~\ref{lem:6gener-case:1} there exists a connected, simply
  connected, nilpotent Lie group $\nilpt$, a lattice $\Gamma <
  \nilpt$, a diffeomorphism $\altdiff\colon G/D \to\nilpt/ \Gamma $
  and an unipotent ergodic affine diffeomorphism $\transl_0\colon
  \nilpt/ \Gamma \to \nilpt/ \Gamma$ such that $ \altdiff \circ \map_0
  = \transl_0 \circ \altdiff$.

  Thus the statement follows from Lemma~\ref{lem:6gener-case:4}.
\end{proof}

%%% Local Variables: 
%%% mode: latex
%%% TeX-master: "InvDist_affine"
%%% End: 

\section{Open problems}
\label{sec:7}

We conclude the paper by stating some (mostly well-known) open
problems and conjectures on the stability and the codimension of
smooth flows.

\begin{conjecture} (A. Katok)
  % \marginpar{\tiny{Would Katok agree with this attribution??}}
  Every homogeneous flow (on a compact homogeneous space)
  % \marginpar{\tiny{finite volume?}}
  which fails to be stable (in the sense that the range of the Lie
  derivative on the space of smooth functions is not closed) projects
  onto a Liouvillean linear flow on a torus. In this case, the flow is
  still stable on the orthogonal complement of the subspace of toral
  functions (in other words, the subspace of all functions with zero
  average along each fiber of the projection).
  % \marginpar{\tiny{we should put here that the toral quotient is
  % assumed to be maximal or something like this}}
\end{conjecture}

As mentioned in the Introduction, hyperbolic and partially hyperbolic,
central isometric (or more generally with uniform sub-exponential
central growth), accessible systems are stable.

In the unipotent case, it is proved in \cite{FF1} that $SL(2,\R)$
unipotent flows (horocycles) on finite volume homogeneous spaces are
stable and in \cite{FF} that the above conjecture holds for nilflows.

\begin{problem} Classify all compact manifolds which admit uniquely
  ergodic flows with $(a)$ a unique invariant distribution (equal to
  the unique invariant measure) up to normalization; $(b)$ a finite
  dimensional space of invariant distributions.
\end{problem}

Example of manifolds (and flows) of type $(a)$ have been found by
A.~Avila, B.~Fayad and A.~Kocsard \cite{AFK}. Note that the Katok
(Greenfield-Wallach) conjecture implies that in all non-toral examples
of type $(a)$ the flow cannot be stable. Recently A.~Avila and
A.~Kocsard \cite{private} have announced that they have constructed
maps on the two-torus having a space of invariant distributions of
arbitrary odd dimension.  It is unclear whether examples of this type
can be stable:

\begin{problem} (M. Herman) Does there exists a stable flow with
  finitely many invariant distributions which is not smoothly
  conjugate to a Diophantine linear flow on a torus?
\end{problem}

The only known example which comes close to an affirmative answer to
this problem is given by generic area-preserving flows on compact
higher genus surfaces \cite{Forni97}, \cite{MMY}.  Such flows are
generically stable and have a finite dimensional space of invariant
distributions in every finite differentiability class (but not in the
class of infinitely differentiable functions).

%%% Local Variables: 
%%% mode: latex
%%% TeX-master: "InvDist_affine"
%%% End: 

\bibliography{InvDist_affine}
\bibliographystyle{amsalpha}
\end{document}